\documentclass{amsart}
\usepackage{animate}
\usepackage{amsmath}
\usepackage{amssymb}
\usepackage{url}
\usepackage{graphics}
\usepackage{subfigure}
\usepackage{graphicx}
\usepackage{textcomp}
\usepackage{mathrsfs}
\usepackage{epstopdf}
\usepackage{algorithmic, algorithm}
\usepackage{color}

\newtheorem{theorem}{Theorem}[section]

\theoremstyle{definition}
\newtheorem{definition}[theorem]{Definition}

\theoremstyle{remark}
\newtheorem{remark}[theorem]{Remark}

\numberwithin{equation}{section}

\begin{document}
\title{Hessian Recovery for Finite Element Methods}

\author{Hailong Guo}
\address{Department of Mathematics, Wayne State University, Detroit,
MI 48202}
\email{guo@math.wayne.edu}
\thanks {This work is supported in part by the US National Science Foundation
through grant 1115530.}

\author{Zhimin Zhang}
\address{Department of Mathematics, Wayne State University, Detroit,
MI 48202}
\email{zzhang@math.wayne.edu}

\author{Ren Zhao}
\address{Department of Mathematics, Wayne State University, Detroit,
MI 48202}
\email{rzhao@math.wayne.edu}





\begin{abstract}
    In this article, we propose  and analyze an effective Hessian recovery
    strategy for the Lagrangian finite element method of arbitrary order.
    We prove  that the proposed Hessian recovery method
    preserves polynomials of degree
    $k+1$ on general unstructured meshes and superconverges at a  rate
    of $O(h^k)$ on mildly structured meshes. In addition, the method
    is proved to be ultraconvergent (two order higher) for translation
  invariant finite element space of any order.
    Numerical examples are presented to support our  theoretical   results.
\end{abstract}

\maketitle

\section{Introduction}
Post-processing is an important technique in scientific computing, where
it is necessary to draw some useful
information that have physical meanings such as velocity, flux,
stress, etc., from the  primary results of the computation.
These  quantities of interest usually involve derivatives of the
primary data.
Some popular post-processing techniques include the celebrated
Zienkiewicz-Zhu superconvergent patch recovery (SPR) \cite{zz1992},
polynomial preserving recovery  (PPR) \cite{zhang2005, naga2005}, and
edge based recovery \cite{pouliot2013}, which were proposed to
  obtain accurate gradients with reasonable cost.
  Similarly, post-processing for second order derivatives, which are
  related to physical quantities such as momentum and Hessian, are
  also desirable.  Hessian matrix is particularly
 significant in adaptive mesh design, since it can indicate the direction where the function changes the most
 and guide us to construct anisotropic meshes to cope with the
 anisotropic properties of
 the solution of the underlying partial differential equation \cite{bank2003,cao2013}.
It also plays an important role in finite element approximation of
second order non-variational elliptic problems \cite{lakkis2011},
numerical solution of some fully nonlinear equations such
as Monge-Amp$\grave{\text{e}}$re equation
\cite{lakkis2013, neilan}, and designing nonlocal
finite element technique \cite{gan}.

 There have been some works in literature on this subject. In 1998, Lakhany-Whiteman  used a simple averaging  twice at edge centers of the regular uniform
 triangular mesh  to produce a superconvergent Hessian  \cite{lakhany1998}.
Later, some other reseachers such as Agouzal et al. \cite{agouzal2002} and
Ovall \cite{ovall2008} also studied Hessian recovery.
Comparsion studies of existing Hessian recovery techniques are found in
Vallet et al. \cite{vallet2007} and Picasso et al. \cite{picasso2011}.
However, there is no systematic theory guarantees  convergence in
general circumstances.
Moreover, there are certain technical difficulties in obtaining rigorous
convergence proof for meshes other than the regular pattern triangular mesh.
In a very recent work, Kamenski-Huang argued that it is not necessary to
have very accurate or even convergent
Hessian in order to obtain a good mesh \cite{Kamenski2012}.

Our current work is not targeted on
 the direction of adaptive mesh refinement; instead, our emphasis is to
 obtain accurate Hessian matrices via recovery techniques.
 We propose an effective Hessian recovery method and establish a solid
 theoretical analysis for such a recovery method.
 Our approach is to apply PPR twice to the primarily  computed data.
This idea is natural. However, the mathematical
theory behind is non-trivial and quite involved, especially
in the ultraconvergence analysis of the recovered Hessian.
A direct calculation of the
gradient from the linear finite element space has linear convergent rate and
the Hessian has no convergence
at all. Our Hessian recovery can achieve second-order convergence
under some uniform meshes, which is a very surprising result!

\section{Preliminaries}
In this section, we first introduce some frequently used notation and
then briefly describe the polynomial preserving recovery (PPR)
operator \cite{zhang2005, naga2005}, which is the
basis of our  Hessian recovery method.
\subsection{Notation}
Let $\Omega$ be a bounded polygonal domain with  Lipschitz boundary
$\partial \Omega$ in $\mathbb{R}^2$. Throughout this article, the standard
notation for classical Sobolev spaces and their associate norms are
adopted as in \cite{brenner2008, ciarlet1978}. A multi-index $\alpha$ is a
$2$-tuple of non-negative integers
$\alpha_i$, $i=1,2$. The length of $\alpha$ is
given by
\begin{equation*}
    |\alpha|=\sum_{i=1}^{2}\alpha_i.
\end{equation*}
For $u\in W^k_p(\Omega)$ and $|\alpha|\le k$, denote $D^{\alpha}u$ the
weak partial derivative $(\frac{\partial}{\partial x})^{\alpha_1}
(\frac{\partial}{\partial y})^{\alpha_2}u$.
Also, $D^{k}u $ with $|\alpha|=k$ is the vector of all
partial derivatives of order $k$. The Hessian operator
$H$ is denoted by
\begin{equation}
    H=
    \begin{pmatrix}
	\partial_{xx}&\partial_{xy}\\
	\partial_{yx}&\partial_{yy}
    \end{pmatrix}.
    \label{secondorder}
\end{equation}
For a subdomain $\mathcal{A}$
of $\Omega$, let $\mathbb{P}_m(\mathcal{A})$ be the space of polynomials of
degree less than or equal to $m$ over $\mathcal{A}$ and $n_m$ be the
dimension of $\mathbb{P}_m(\mathcal{A})$ with $n_m=\frac{1}{2}(m+1)(m+2)$.
$W^k_p(\mathcal{A})$ denotes the classical Sobolev space with norm
$\|\cdot\|_{k, p, \mathcal{A}} $ and seminorm $|\cdot|_{k, p, \mathcal{A}}$.
 When $p = 2$, we denote simply $H^{k}(\mathcal{A}) =
W^{k}_{2}(\mathcal{A})$ and the subscript $p$ is omitted.

For any $0<h< \frac{1}{2}$,
let $\mathcal{T}_h$ be a shape regular triangulation of $\bar{\Omega}$
 with mesh size at most $h$, i.e.
\begin{equation*}
\bar{\Omega}=\bigcup_{K\in\mathcal{T}_h} K,
\end{equation*}
where $K$  is a triangle. For any $k\in \mathbb{N}$, define the
continuous finite element space $S_h$ of order $k$ as
\begin{equation*}
    S_h=\{v\in C(\bar{\Omega}): v|_{K}\in \mathbb{P}_k(K),
    \quad \forall K\in \mathcal{T}_h\}\subset H^1(\Omega).
\end{equation*}
Let $\mathcal{N}_h$ denote the set of mesh nodes, i.e.
the dual space of $S_h$. The  standard Lagrange basis of $S_h$ is denoted by
 $\{\phi_z: z\in \mathcal{N}_h\}$
with $\phi_z(z')=\delta_{zz'}$ for all $z, z'\in \mathcal{N}_h$.
 For any $v\in H^1(\Omega)\cap
C(\Omega)$, let $v_I$ be the interpolation of $v$ in $S_h$,
i.e.,$v_I=\sum\limits_{z\in\mathcal{N}_h}v(z)\phi_z$.

For $\mathcal{A}\subset \Omega$, let $S_h(\mathcal{A})$ denote the restrictions
of functions in $S_h$ to $\mathcal{A}$ and let $S_h^{\text{comp}}(\mathcal{A})$ denote the set of those functions
in $S_h(\mathcal{A})$ with compact support in the interior of $\mathcal{A}$
\cite{wahlbin1995}. Let
$ \Omega_0\subset\subset \Omega_1 \subset\subset \Omega_2 \subset\subset\Omega$
be  separated by $d\ge c_oh$ and  $\ell$ be a direction, i.e., a unit vector
in $\mathbb{R}^2$.
Let $\tau$ be a parameter, which will typically be a multiply of $h$.
Let $T^{\ell}_{\tau}$ denote translation by $\tau$ in the direction $\ell$,
i.e.,
\begin{equation}
  T^{\ell}_{\tau}v(x) = v(x+\tau\ell),
  \label{equ:translation}
\end{equation}
and for an integer $\nu$
\begin{equation}
  T^{\ell}_{\nu\tau}v(x) = v(x+\nu\tau\ell).
  \label{trans}
\end{equation}
Following the definition of \cite{wahlbin1995}, the finite element space
$S_h$ is called translation invariant by $\tau$ in the direction $\ell$ if
\begin{equation}
  T^{\ell}_{\nu \tau}v\in S^{\text{comp}}_h(\Omega), \quad
  \forall v \in S^{\text{comp}}_h(\Omega_1),
\end{equation}
for some integer $\nu$ with $|\nu| < M$.  Equivalently,
$\mathcal{T}_h$ is called a translation invariant mesh. To clarify the matter,
we consider   five popular triangular mesh patterns: Regular, Chevron,
Union-Jack, Criss-cross, and equilateral patterns, as shown in
Figure \ref{mesh}.

\begin{figure}[!h]
  \centering
  \subfigure[]{
  \begin{minipage}[c]{0.2\textwidth}
  \centering
  \includegraphics[width=0.9\textwidth]{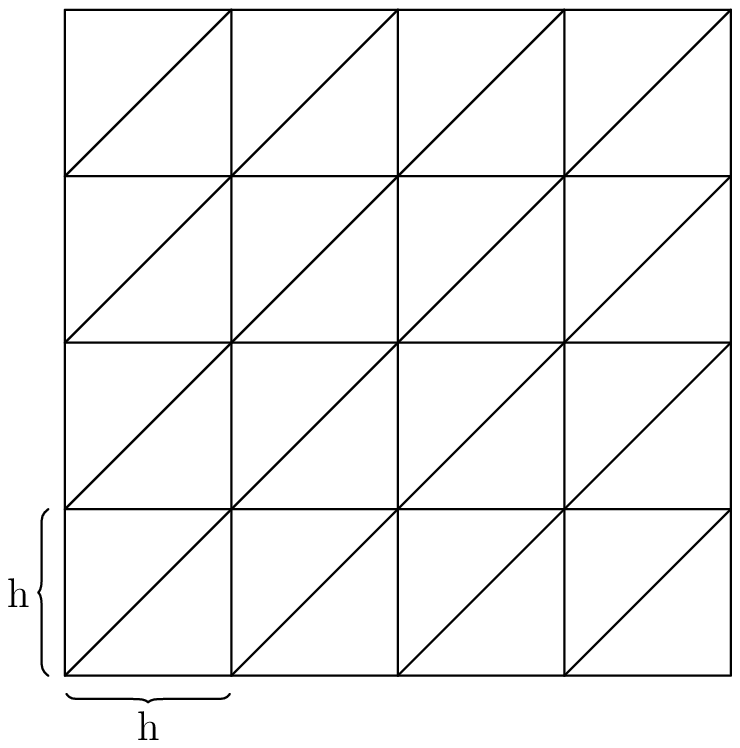}
\label{regularmesh}
\end{minipage}}%
\subfigure[]{
 \begin{minipage}[c]{0.2\textwidth}
  \centering
  \includegraphics[width=0.9\textwidth]{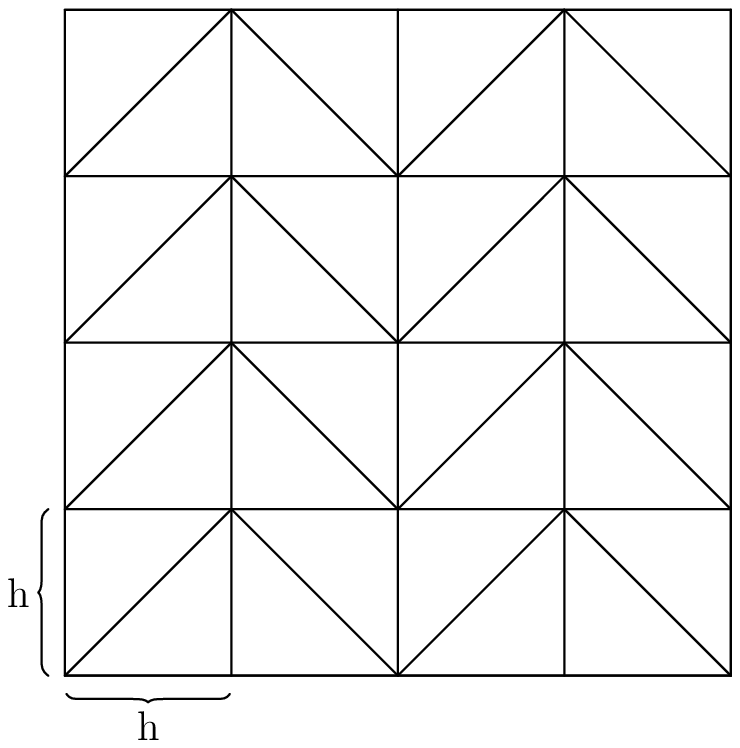}
\label{chevronmesh}
\end{minipage}}%
\subfigure[]{
\begin{minipage}[c]{0.2\textwidth}
  \centering
  \includegraphics[width=0.9\textwidth]{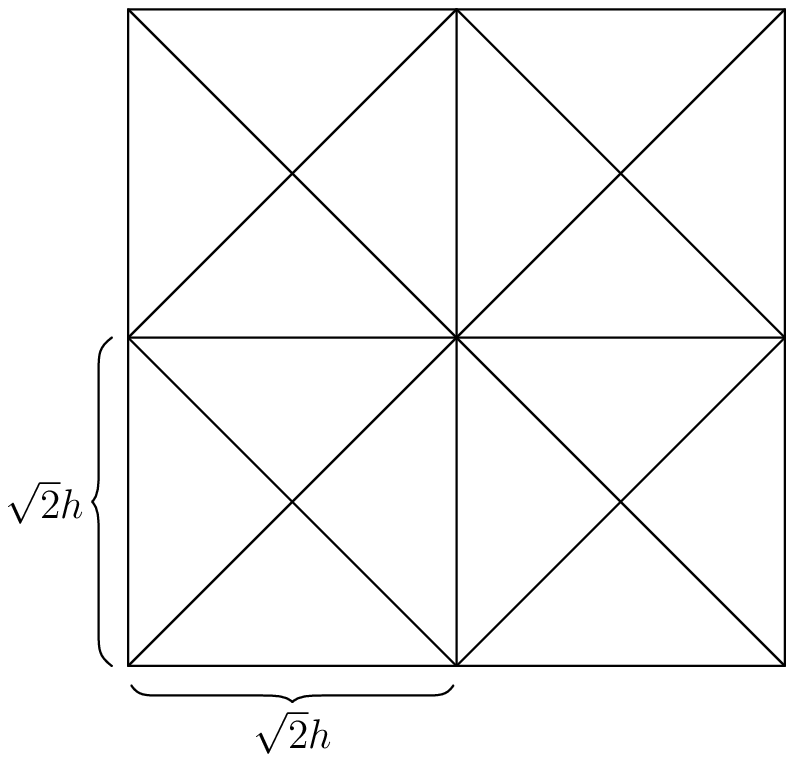}
\label{crisscrossmesh}
\end{minipage}}%
\subfigure[]{
\begin{minipage}[c]{0.2\textwidth}
  \centering
  \includegraphics[width=0.9\textwidth]{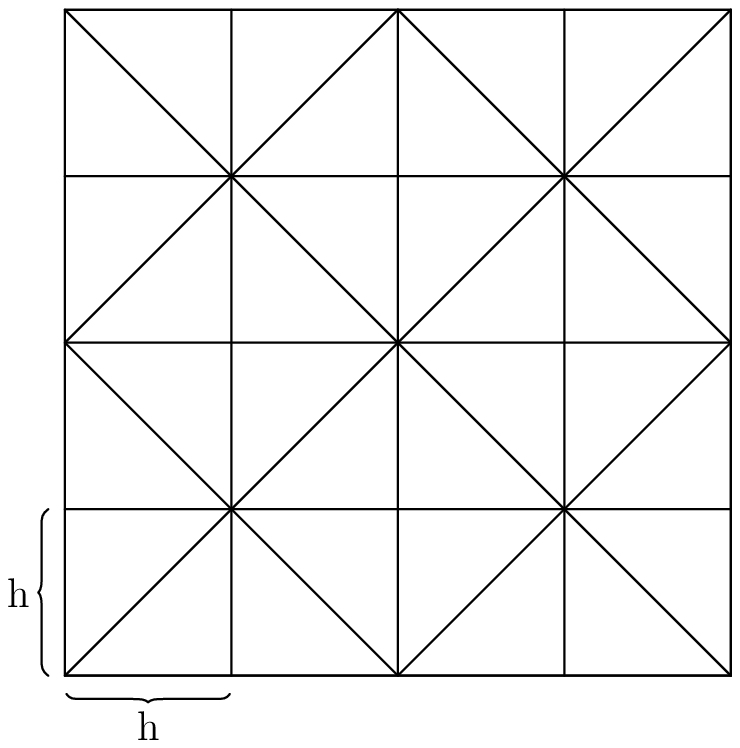}
\label{unionjackmesh}
\end{minipage}}%
\subfigure[]{
\begin{minipage}[c]{0.2\textwidth}
  \centering
  \includegraphics[width=0.9\textwidth]{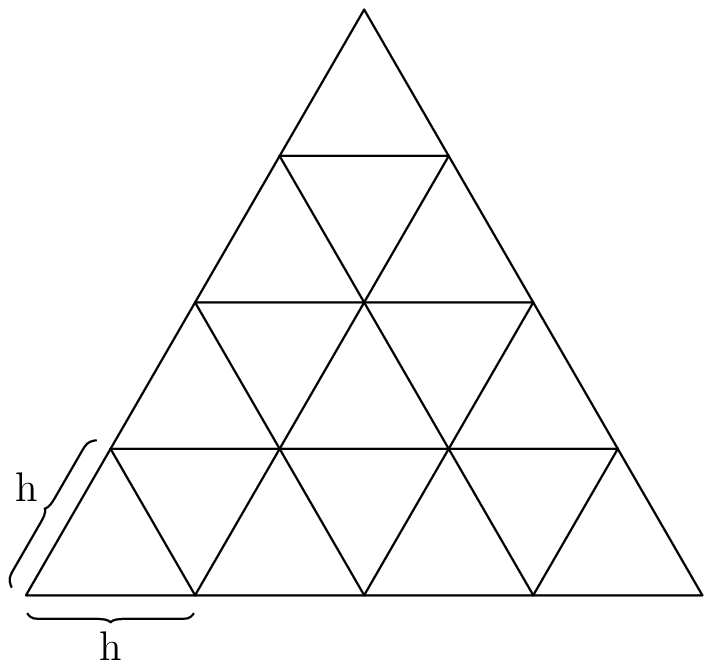}
\label{equallatermesh}
\end{minipage}}
  \caption{Five types of uniform meshes: (a) Regular pattern;
  (b) Chevron pattern; (c) Criss-cross pattern; (d) Union-Jack pattern;
  (e) Equilateral pattern}
  \label{mesh}
\end{figure}
We see that:

1) Regular pattern is translation invariant by $h$ in directions $(1,0)$ and $(0,1)$,
by $2\sqrt{2}h$ in directions $(\pm \frac{\sqrt{2}}{2},\frac{\sqrt{2}}{2})$,
and by $\sqrt{5}h$ in directions
$(\frac{2\sqrt{5}}{5},\pm \frac{\sqrt{5}}{5})$ and $(\pm \frac{\sqrt{5}}{5},\frac{2\sqrt{5}}{5})$, ......

2) Chevron pattern is translation invariant by $h$ in the direction $(0,1)$,
by $2h$ in the direction $(1,0)$, and by $2\sqrt{2}h$ in directions
$(\pm \frac{\sqrt{2}}{2},\frac{\sqrt{2}}{2})$,
and by $\sqrt{5}h$ in directions $(\pm \frac{\sqrt{5}}{5},\frac{2\sqrt{5}}{5})$, ......

3) Criss-cross pattern is translation invariant by $\sqrt{2}h$ in directions $(1,0)$ and $(0,1)$,
and by $2h$ in directions
$(\pm \frac{\sqrt{2}}{2},\frac{\sqrt{2}}{2})$, ......

4) Union-Jack pattern is translation invariant by $2h$ in directions $(1,0)$ and $(0,1)$,
and by $2\sqrt{2}h$ in directions
$(\pm \frac{\sqrt{2}}{2},\frac{\sqrt{2}}{2})$, ......

5) Equilateral pattern is translation invariant by $h$ in directions $(1,0)$ and $(\pm \frac{1}{2},\frac{\sqrt{3}}{2})$,
and by $\sqrt{3}h$ in directions $(0,1)$ and $(\frac{\sqrt{3}}{2},\pm \frac{1}{2})$, ......

Throughout  this article, the letter $C$ or $c$, with or without subscript,
denotes  a generic constant which is independent of $h$ and may not be the same
at each occurrence. To simplify notation,  we  denote $x\le Cy$ by $x\lesssim y$.

\subsection{Polynomial preserving recovery}
Let $G_h: S_h\rightarrow S_h\times S_h$ be the PPR operator.
Given a function $u_h\in S_h$, it suffices to define $(G_hu_h)(z)$
for all $z\in \mathcal{N}_h$.
Let $z\in\mathcal{N}_h$ be a vertex and $\mathcal{K}_z$ be a patch of elements
around $z$ which is defined in \cite{zhang2005, naga2005}.
Select all nodes in $\mathcal{N}_h\cap \mathcal{K}_z$ as sampling
points  and fit a polynomial $p_z\in \mathbb{P}_{k+1}(\mathcal{K}_z)$ in  the least
squares sense at those sampling points, i.e.
\begin{equation}
  p_z=\arg \min_{p\in \mathbb{P}_{k+1}(\mathcal{K}_z)}
    \sum_{\tilde{z}\in\mathcal{N}_h\cap \mathcal{K}_z}(u_h-p)^2(\tilde{z}).
    \label{least}
\end{equation}
Then the recovered gradient at $z$ is defined as
\begin{equation*}
    (G_hu_h)(z) = \nabla p_z(z).
\end{equation*}
For linear element, all nodes in $\mathcal{N}_h$ are vertices and hence
$G_hu_h$ is well defined. However, $\mathcal{N}_h$ may contain
edge nodes or interior nodes for higher order elements.
If  $z$ is an edge node which lies on an edge between
two vertices $z_1$ and $z_2$, we define
\begin{equation*}
    (G_hu_h)(z)=\beta\nabla p_{z_1}(z) + (1-\beta)\nabla p_{z_2}(z)
\end{equation*}
where $\beta$ is determined by the ratio of distances of $z$ to
$z_1$ and $z_2$. If $z$ is an interior node which lies in a triangle
formed by three vertices $z_1$, $z_2$, and $z_3$, we define
\begin{equation*}
    (G_hu_h)(z)=\sum_{j=1}^{3}\beta_j\nabla p_{z_j}(z),
\end{equation*}
where $\beta_j$ is the barycentric coordinate of $z$.
\begin{remark}
     It was proved in \cite{naga2004} that certain rank condition and geometric condition
    guarantee the uniqueness of $p_z$ in \eqref{least}.
\end{remark}
\begin{remark}
    In order to avoid numerical instability,  a discrete least squares fitting
    process is carried out on a reference  patch $\omega_{z}$ .
\end{remark}
\section{Hessian recovery method}
Given $u\in S_h$, let $G_hu \in S_h\times S_h$ be the recovered gradient
using PPR as defined in previous section. We rewrite $G_hu$ as
\begin{equation}
    G_hu=
    \begin{pmatrix}
	G_h^xu\\
	G_h^yu
    \end{pmatrix}.
    \label{rgrad}
\end{equation}
In order to recover the Hessian  matrix of $u$, we apply gradient recovery
operator $G_h$ to $G_h^xu$ and $G_h^yu$ one more time, respectively, and
define the Hessian recovery operator $H_h$ as follows
\begin{equation}
    H_hu=
    \begin{pmatrix} G_h(G_h^xu),&G_h(G_h^yu)
    \end{pmatrix}
    =
    \begin{pmatrix}
	G_h^x(G_h^xu)&G_h^x(G_h^yu)\\
	G_h^y(G_h^xu)&G_h^y(G_h^yu)
    \end{pmatrix}.
    \label{hessian}
\end{equation}
Just as PPR, we obtain $H_h: S_h\rightarrow  S_h^2\times S_h^2$
on the whole domain $\Omega$ by interpolation after determining values of
$H_hu$ at all nodes in $\mathcal{N}_h$.
\begin{remark}
    The two gradient  recovery operators  in definition \eqref{hessian} of
    $H_h$ can be different. Actually we can define
 the Hessian recovery operator $H_h$ as following
\begin{equation*}
    H_hu=
    \begin{pmatrix} \tilde{G}_h(G_h^xu),&\tilde{G}_h(G_h^yu)
    \end{pmatrix}.
\end{equation*}
By choosing $G_h$ and $\tilde{G}_h$ as PPR or SPR operator,  we  obtain
four different Hessian recovery operators, i.e., PPR-PPR, PPR-SPR, SPR-PPR, and
SPR-SPR. However, numerical tests  have shown  that PPR-PPR is the best one.
\label{fourmethods}
\end{remark}

In order to demonstrate our method, we shall discuss two examples in detail.
For the sake of simplicity, only linear element on
uniform meshes will be considered. In practice, the method can be applied to
arbitrary meshes and higher order elements.

{\em Example 1}. Consider the regular pattern uniform mesh as in
Figure \ref{regular}. We want to recovery the Hessian
matrix at $z_0$.  As deduced in \cite{zhang2005}, the recovered gradient
at $z_0$ is given by
\begin{equation*}
  (G_hu)(z_0)=\frac{1}{6h}\left( \begin{pmatrix}2\\ 1\end{pmatrix}u_1 +
    \begin{pmatrix}1\\ 2\end{pmatrix}u_2+
      \begin{pmatrix}-1\\ 1\end{pmatrix}u_3 +
    \begin{pmatrix}-2\\ -1\end{pmatrix}u_4 +
      \begin{pmatrix}-1\\ -2\end{pmatrix}u_5+
      \begin{pmatrix}1\\ -1\end{pmatrix}u_6\right).
\end{equation*}
Here $u_i=u(z_i), (i = 0, 1, \ldots 18)$ represents function value of $u$ at node $z_i$.
Thus, according to the definition \eqref{hessian} of  the Hessian recovery operator $H_h$,
we have
 \begin{equation}
     \begin{split}
      \begin{pmatrix}H^{xx}_hu\\H^{xy}_hu\end{pmatrix}(z_0)
	=\frac{1}{6h}\left(2(G_hu)(z_1)+(G_hu)(z_2)-(G_hu)(z_3)-\right. \\
	\left. 2(G_hu)(z_4)
	-(G_hu)(z_5)+(G_hu)(z_6)\right),
     \end{split}
	\label{hessianx}
    \end{equation}
    and
\begin{equation}
     \begin{split}
      \begin{pmatrix}H^{yx}_hu\\H^{yy}_hu\end{pmatrix}(z_0)
	=\frac{1}{6h}\left((G_hu)(z_1)+2(G_hu)(z_2)+(G_hu)(z_3)-\right. \\
	\left. (G_hu)(z_4)
	-2(G_hu)(z_5)-(G_hu)(z_6)\right),
     \end{split}
	\label{hessiany}
    \end{equation}
where
\begin{align*}
(G_hu)(z_1)=\frac{1}{6h}\left( \begin{pmatrix}2\\ 1\end{pmatrix}u_7 +
    \begin{pmatrix}1\\ 2\end{pmatrix}u_8+
      \begin{pmatrix}-1\\ 1\end{pmatrix}u_2 +
    \begin{pmatrix}-2\\ -1\end{pmatrix}u_0 +
	\begin{pmatrix}-1\\ -2\end{pmatrix}u_{18}+
      \begin{pmatrix}1\\ -1\end{pmatrix}u_6\right),
\end{align*}
and $(G_hu)(z_2), \ldots, (G_hu)(z_6)$  follow the similar pattern.
Direct calculation  reveals that
\begin{align*}
  (H_h^{xx}u)(z_0)=\frac{1}{36h^2}(&-12u_0+2u_1-4u_2-4u_3+2u_4-4u_5-4u_6
   +4u_7 + 4u_8 +u_9\\&-2u_{10}+u_{11}+4u_{12}+4u_{13}+
  4u_{14}+u_{15}-2u_{16}+u_{17}+4u_{18}),\\
  (H_h^{xy}u)(z_0)=\frac{1}{36h^2}(&6u_0-u_1+5u_2-u_3-u_4+5u_5-
  u_6-2u_7+u_8+u_9\\&+u_{10}-2u_{11}-5u_{12}-2u_{13}+u_{14}+u_{15}
  +u_{16}-2u_{17}-5u_{18}),\\
  (H_h^{yx}u)(z_0)=\frac{1}{36h^2}(&6u_0-u_1+5u_2-u_3-u_4+5u_5-
  u_6-2u_7+u_8+u_9\\&+u_{10}-2u_{11}-5u_{12}-2u_{13}+u_{14}+u_{15}
  +u_{16}-2u_{17}-5u_{18}),\\
   (H_h^{yy}u)(z_0)=\frac{1}{36h^2}(&-12u_0-4u_1-4u_2+2u_3-4u_4-4u_5+2u_6
   +u_7 - 2u_8 +u_9\\&+4u_{10}+4u_{11}+4u_{12}+u_{13}-
  2u_{14}+u_{15}+4u_{16}+4u_{17}+4u_{18}).
\end{align*}
It is observed that $(H^{xy}_hu)(z_0)=(H^{yx}_hu)(z_0)$,
which means  the recovered
Hessian matrix is symmetric, a property of the exact Hessian
we would like to maintain.

Using  Taylor expansion, we can show that
\begin{align*}
  &(H_h^{xx}u)(z_0)=u_{xx}(z_0)+\frac{h^2}{3}(u_{xxxx}(z_0)+
  u_{xxxy}(z_0)+u_{xxyy}(z_0))+O(h^4),\\
  &(H_h^{xy}u)(z_0)=u_{xy}(z_0)+\frac{h^2}{3}(u_{xxxy}(z_0)+
  u_{xxyy}(z_0)+u_{xyyy}(z_0))+O(h^4),\\
  &(H_h^{yx}u)(z_0)=u_{yx}(z_0)+\frac{h^2}{3}(u_{xxxy}(z_0)+
  u_{xxyy}(z_0)+u_{xyyy}(z_0))+O(h^4),\\
  &(H_h^{yy}u)(z_0)=u_{yy}(z_0)+\frac{h^2}{3}(u_{xxyy}(z_0)+
  u_{xyyy}(z_0)+u_{yyyy}(z_0))+O(h^4),
\end{align*}
which imply that $H_hu$ provides a second order approximation of $Hu$ at $z_0$.

{\em Example 2.} Consider the
Chevron pattern uniform mesh as shown in  Figure \ref{chevron}.
Repeating the procedure as in Example 1, we derive the recovered
Hessian matrix at $z_0$ as
\begin{align*}
    (H_h^{xx}u)(z_0)=\frac{1}{144h^2}(&-72u_{0}+36u_{13}+36u_7),\\
    (H_h^{xy}u)(z_0)=\frac{1}{144h^2}(&-12u_1+12u_3+24u_4-24u_6+6u_7+
    \\&+36u_9-36u_{11}-6u_{13}+6u_{14}-6u_{18}),\\
    (H_h^{yx}u)(z_0)=\frac{1}{144h^2}(&12u_1-12u_3+36u_4-36u_6-6u_7+
    \\&6u_8+24u_9-24u_{11}-6u_{12}+6u_{13}),\\
   (H_h^{yy}u)(z_0)=\frac{1}{144h^2}(&-48u_0-10u_1-22u_2-10u_3-10u_4+18u_5-
   \\&10u_6-2u_7+u_8+10u_9+36u_{10}+10u_{11}+u_{12}-\\
   &2u_{13}+u_{14}+10u_{15}+16u_{16}
   +10u_{17}+u_{18}).
\end{align*}
In addition, we have the following Taylor expansion
\begin{align*}
    &(H_h^{xx}u)(z_0)=u_{xx}(z_0)+\frac{h^2}{3}u_{xxxx}(z_0)+
    \frac{2h^4}{45}u_{xxxxxx}(z_0)+O(h^5),\\
  &(H_h^{xy}u)(z_0)=u_{xy}(z_0)+\frac{h^2}{12}(3u_{xxxy}(z_0)+
  2u_{xyyy}(z_0))-\frac{h^3}{24}u_{xxxyy}(z_0)+O(h^4),\\
  &(H_h^{yx}u)(z_0)=u_{yx}(z_0)+\frac{h^2}{12}(3u_{xxxy}(z_0)+
  2u_{xyyy}(z_0))+\frac{h^3}{24}u_{xxxyy}(z_0)+O(h^4),\\
  &(H_h^{yy}u)(z_0)=u_{yy}(z_0)+\frac{h^2}{6}(u_{xxyy}(z_0)+
  2u_{yyyy}(z_0))-\frac{5h^3}{72}u_{xxyyy}(z_0)+O(h^4).
\end{align*}
We conclude that $H_hu$ is a second order approximation to the
Hessian matrix. It is worth pointing out that, though $H_h^{xy} \neq H_h^{yx}$
for the Chevron pattern uniform mesh, they are both second order finite
difference schemes at $z_0$.
\begin{remark}
    PPR-PPR is the only one among the four Hessian recovery
    methods mentioned in Remark
    \ref{fourmethods} that provides second order approximation
    for all five  mesh patterns, especially the Chevron pattern.
\end{remark}

\begin{figure}[!h]
  \centering
  \begin{minipage}[c]{0.5\textwidth}
  \centering
  \includegraphics[width=0.9\textwidth]{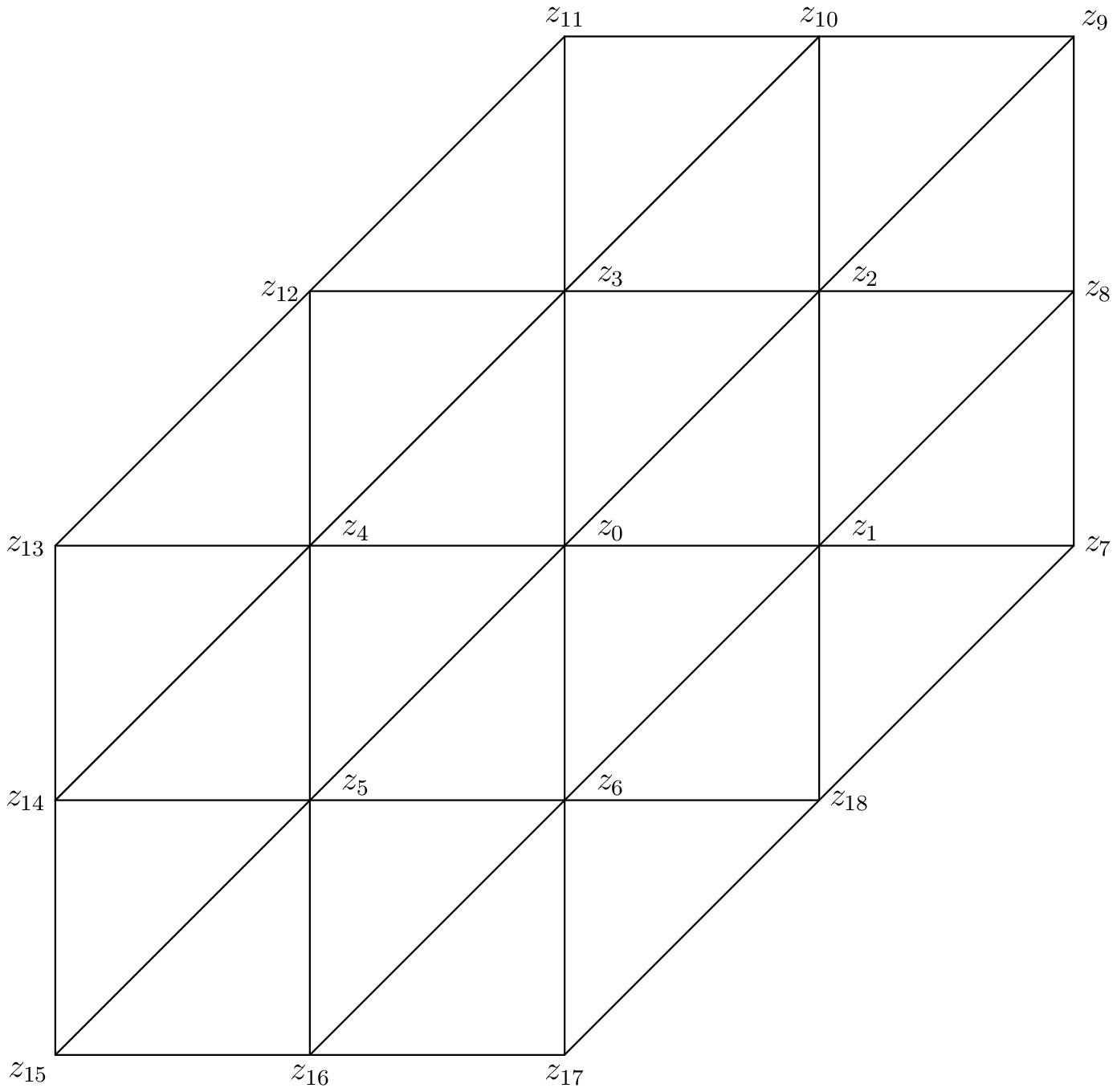}
  \caption{Regular Pattern}
\label{regular}
\end{minipage}%
 \begin{minipage}[c]{0.5\textwidth}
  \centering
  \includegraphics[width=0.9\textwidth]{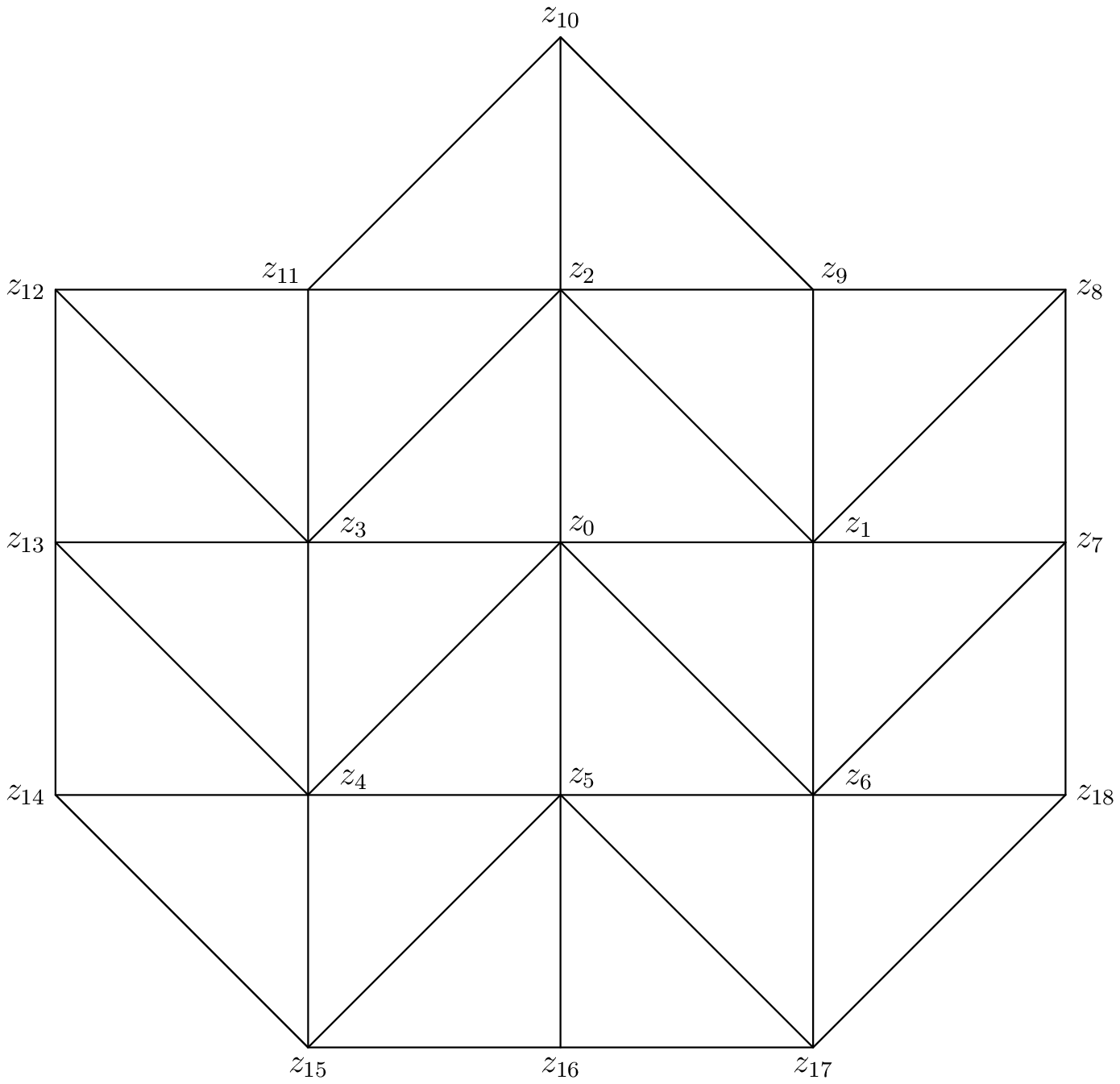}
   \caption{Chevron Pattern}
\label{chevron}
\end{minipage}
\end{figure}

Both example 1 and 2   indicate that
for linear element the PPR-PPR approach is equivalent to a finite difference
scheme of second order accuracy at vertex $z_0$.
In general, we can show that
$H_h$ preserves  polynomials of degree up to $k+1$ for $k$th order element.

Consider $P_k$-element. Let $u$ be a polynomial of degree $k+1$.
Since $G_h$ preserves polynomials of degree $k+1$,
it follows that $G_h u = \nabla u$ which is a polynomial
of degree $k$. Therefore, we have
\begin{equation}
  H_hu = (G_h(G^{x}_hu), G_h(G^{y}_hu))
=(G_h\frac{\partial u}{\partial x}, G_h\frac{\partial u}{\partial x})
=(\nabla\frac{\partial u}{\partial x}, \nabla\frac{\partial u}{\partial x})
= Hu.
  \label{equ:general}
\end{equation}
It means that $H_h$ preserves polynomials of degree $k+1$ for arbitrary mesh.

Now we proceed  translation invariant mesh.
Under the polynomial preserving property,
the recovered gradient is exact for polynomials of degree $k+1$. Therefore
\begin{equation}
G_h^x u = D_x u + h^{k+1} \pmb a^x \cdot D^{k+2}u + h^{k+2} \pmb b^x
\cdot D^{k+3}u + h^{k+3} \pmb c^x \cdot D^{k+4}u + \cdots;
\label{gxk+1}
\end{equation}
\begin{equation}
G_h^y u = D_y u + h^{k+1} \pmb a^y \cdot D^{k+2}u + h^{k+2} \pmb b^y
\cdot D^{k+3}u + h^{k+3} \pmb c^y \cdot D^{k+4}u + \cdots.
\label{gyk+1}
\end{equation}
Note that $\pmb a^x, \pmb a^y, \pmb b^x, \pmb b^y, \pmb c^x, \pmb c^y, \cdots$
are functions of $(x,y)$ if $z = (x, y)$ a nodal point of arbitrary mesh.

Let $\pmb z=(x, y)$ be any node on a translation invariant mesh.  We further assume that
$\pmb z$ is a local symmetry center for all sampling points involved.
Notice that coefficients $\pmb a^x$, $\pmb a^y$, $\pmb b^x$, $\pmb b^y, \ldots$
depend only on the coordinates of nodes, since we recover gradient at
nodes only.  Thus for translation invariant meshes,
$\pmb a^x$, $\pmb a^y$, $\pmb b^x$, $\pmb b^y, \ldots$  are constants.
 In addition, due to symmetry, it makes no difference if we perform $G^x_h$ or
$G^y_h$ first.
Hence,
\begin{equation}
    \begin{split}
	&(H_h^{xy}u)(\pmb z)  = (G_h^y(G_h^x u))(\pmb z) \\
	=& G_h^y [ D_x u(\pmb z) + h^{k+1} \pmb a^x \cdot D^{k+2}u(\pmb z)
	    + h^{k+2} \pmb b^x \cdot D^{k+3}u(\pmb z) + \cdots ] \\
        =& (G_h^y(D_x u))(\pmb z) + h^{k+1} (\pmb a^x \cdot
	G_h^y(D^{k+2}u) )(\pmb z) + h^{k+2}
	(\pmb b^x \cdot G_h^y(D^{k+3}u) )(\pmb z) + \cdots  \\
        =& (D_y D_xu)(\pmb z) + h^{k+1} (\pmb a^y \cdot D^{k+2} D_xu)(\pmb z)
	+ h^{k+2} (\pmb b^y \cdot D^{k+3}D_xu )(\pmb z)\\
        & + h^{k+1} (\pmb a^x \cdot D_y(D^{k+2}u) )(\pmb z) + h^{k+2}(\pmb b^x
	\cdot D_y(D^{k+3}u) )(\pmb z)+ O(h^{k+3})  \\
        =& (D_y D_xu)(\pmb z) + h^{k+1} [\pmb a^y \cdot D^{k+2} D_xu + \pmb a^x
	    \cdot D_y(D^{k+2}u)](\pmb z)+\\
          &  h^{k+2} [ \pmb b^y \cdot D^{k+3}D_xu + \pmb b^x
	      \cdot D_y(D^{k+3}u) ](\pmb z) + O(h^{k+3}).
    \end{split}
    \label{tgxyk+1}
\end{equation}
Notice that \eqref{tgxyk+1} is valid only at
nodal points. Similarly,

\begin{align}
    \begin{split}
	(H_h^{yx}u)(\pmb z) =& (D_x D_yu)(\pmb z) +
	h^{k+1} [\pmb a^x \cdot D^{k+2} D_yu +
	    \pmb a^y \cdot D_x(D^{k+2}u)](\pmb z)+ \\
            & h^{k+2} [ \pmb b^x \cdot D^{k+3}D_yu +\pmb b^y
		\cdot  D_x(D^{k+3}u) ](\pmb z) + O(h^{k+3});
    \end{split}
    \label{tgyxk+1}\\
    \begin{split}
	(H_h^{xx}u)(\pmb z) =& (D_x D_xu)(\pmb z) + h^{k+1}
	[\pmb a^x \cdot D^{k+2} D_xu
	    + \pmb a^x \cdot D_x(D^{k+2}u)](\pmb z)+ \\
          &h^{k+2} [ \pmb b^x \cdot D^{k+3}D_xu + \pmb b^x
	      \cdot D_x(D^{k+3}u) ](\pmb z) + O(h^{k+3});
    \end{split}
    \label{tgxxk+1}\\
    \begin{split}
        (H_h^{yy}u)(\pmb z) =& (D_y D_yu)(z) +
	h^{k+1} [\pmb a^y \cdot D^{k+2} D_yu +
	    \pmb a^y \cdot D_y(D^{k+2}u)](\pmb z)+ \\
        &h^{k+2} [ \pmb b^y \cdot D^{k+3}D_yu + \pmb b^y \cdot
	D_y(D^{k+3}u)](\pmb z)	+ O(h^{k+3}).
    \end{split}
    \label{tgyyk+1}
\end{align}
\eqref{tgxyk+1}--\eqref{tgyyk+1}  imply that the Hessian recovery operator
$H_h$ is exact for polynomials of degree  $k+2$ for translation
invariant meshes. Also, we observe  $H_h^{xy}=H_h^{yx}$  from \eqref{tgxyk+1}
and \eqref{tgyxk+1}.

It is worth pointing out that, except for the Chevron pattern,
\eqref{tgxyk+1}--\eqref{tgyyk+1} are valid for the other four patterns
of uniform meshes, since the recovered gradient $G_hu$ produces the same stencil
at each node.

Next we consider even order ($k=2r)$ element on translation invariant
meshes, in which case
\begin{equation}
\pmb a^x(\pmb z) = \pmb 0, \quad \pmb c^x(\pmb z) = \pmb 0,
\quad \pmb a^y(\pmb z) = \pmb 0, \quad \pmb c^y(\pmb z) = \pmb 0;
\label{ff}
\end{equation}
\begin{equation}
D \pmb a^x(\pmb z) = \pmb 0, \quad D \pmb c^x(\pmb z) = \pmb 0,
\quad D \pmb a^y(\pmb z) = \pmb 0, \quad D \pmb c^y(\pmb z) = \pmb 0.
\label{dd}
\end{equation}
and $\pmb b^x , \pmb b^y, \cdots $ are constants in \eqref{gyk+1}.
 Here the symbol $D$ is understood as taking all partial derivatives
 to each entry of the vector. Consequently,
\begin{equation}\label{gyk+2}
(G_h^y u)(\pmb z) = (D_y u)(\pmb z) + h^{k+2}
(\pmb b^y \cdot D^{k+3}u)(\pmb z) + O(h^{k+4}),
\end{equation}
Also, \eqref{gyk+2} is valid only at nodal points.
Plugging \eqref{gxk+1} into \eqref{gyk+2} yields
\begin{equation*}
    \begin{split}
	&(H_h^{xy} u)(\pmb z) = (G_h^y G_h^x u)(\pmb z) \\
      = & (D_y G_h^x u)(\pmb z) + h^{k+2} (\pmb b^y \cdot D^{k+3}
      G_h^x u)(\pmb z) + O(h^{k+4}) \\
      = & D_y ( D_x u + h^{k+1} \pmb a^x \cdot D^{k+2}u + h^{k+2}
      \pmb b^x \cdot D^{k+3}u + h^{k+3} \pmb c^x \cdot D^{k+4}u \\
      &+\cdots )(\pmb z) + h^{k+2} (\pmb b^y \cdot D^{k+3} D_x u)
      (\pmb z) + O(h^{k+4}) \\
     = & (D_y D_x u)(\pmb z) + h^{k+2} (\pmb b^x \cdot  D_yD^{k+3}u +
     \pmb b^y \cdot D^{k+3} D_x u )(\pmb z) + O(h^{k+4}).
    \end{split}
\end{equation*}
In the last identity we have used (\ref{ff}) and (\ref{dd}).

The argument for the other three entries of recovered Hessian matrix are
similar. We conclude that the Hessian recovery operator $H_h$ is exact for
polynomials of degree up to $k+3$ when $k$ is even and the mesh is
translation invariant  and symmetric with respect to $x$ and $y$.

The above results can be summarized as the following theorem:
\begin{theorem}
    The Hessian recovery operator $H_h$ preserves polynomials of degree $k+1$
    for an arbitrary mesh. If $z$ is a node of a translation invariant mesh,
    then $H_h$ preserves polynomials of degree $k+2$  for odd $k$,
    and of degree $k+3$ for even $k$. Moreover, if the sampling points are symmetric with
    respect to $x$ and $y$, then $H_h$ is symmetric.
    \label{poly}
\end{theorem}

\begin{remark}
  According to \cite{vallet2007}, the best  Hessian recovery
  method in the literature
     preserves polynomial of degree 2 for linear element. Our method
    preserves polynomial of degree 2  on general
    unstructured meshes and preserves polynomials of degree 3  on
    translation invariant meshes for linear element.
\end{remark}

\begin{theorem}
    \label{pp}
    Let $u\in W_{\infty}^{k+2}(\mathcal{K}_z)$; then
    \begin{equation*}
      \|H u-H_hu\|_{0, \infty, \mathcal{K}_z}\lesssim h^{k}|u|
	_{k+2, \infty,\mathcal{K}_z}.
    \end{equation*}
    If $z$ is a node of translation invariant mesh and $u\in W_{\infty}^{k+3}(\mathcal{K}_z)$, then
    \begin{equation*}
	|(H u-H_hu)(z)|\lesssim h^{k+1}|u|
	_{k+3, \infty,\mathcal{K}_z}.
    \end{equation*}
Furthermore, if $z$ is a node of translation invariant mesh and
$u\in W_{\infty}^{k+4}(\mathcal{K}_z)$  with $k$ an even number, then
    \begin{equation*}
	|(H u-H_hu)(z)|\lesssim h^{k+2}|u|
	_{k+4, \infty,\mathcal{K}_z}.
    \end{equation*}
\end{theorem}
\begin{proof}
    It is a direct result of Theorem \ref{poly} and
    application of the Hilbert-Bramble Lemma.
\end{proof}

\section{Superconvergence analysis}
In this section, we first use the supercloseness between the gradient of the
finite element solution $u_h$  and the gradient of the interpolation $u_I$
\cite{bank2003, cao2013, huang2013, huang2008, wu2007, xu2004},
and properties of the PPR operator
\cite{zhang2005, naga2004} to establish the superconvergence property
of our Hessian recovery operator  on mildly structured mesh.
 Then we utilize the tool of superconvergence by difference
quotients from \cite{wahlbin1995} to prove the proposed Hessian recovery
method is ultraconvergent
for translation invariant finite element space of any order.

In this section, we consider the following variational problem:
find $u\in H^1(\Omega)$ such that
\begin{equation}
    B(u, v)=\int_{\Omega}(\mathcal{D}\nabla u+\mathbf{b}u)\cdot \nabla v+cuvdx
    =(f, v), \quad \forall v\in H^1(\Omega).
    \label{variational}
  \end{equation}
Here $\mathcal{D}$ is  a $2\times 2$ symmetric positive definite matrix,
$\mathbf{b}$ is a vector, and $c$ as well as $f$ are scalars.
 All coefficient functions are
assumed to be smooth.

In order to insure  \eqref{variational} has a unique solution, we assume
the bilinear form $B(\cdot, \cdot)$ satisfies the continuity condition
\begin{equation}
  |B(u, v)|\le \nu \|u\|_{1, \Omega}\|v\|_{1, \Omega},
  \label{equ:continous}
\end{equation}
for all $u, v \in H^{1}(\Omega)$. We also assume the inf-sup conditions
\cite{ciarlet1978, brenner2008, bank2003}
\begin{equation}
  \inf_{u\in H^1(\Omega)}\sup_{v\in H^1(\Omega)}
    \frac{B(u, v)}{\|u\|_{1, \Omega}\|v\|_{1, \Omega}}
    =
  \sup_{u\in H^1(\Omega)}\inf_{v\in H^1(\Omega)}
    \frac{B(u, v)}{\|u\|_{1, \Omega}\|v\|_{1, \Omega}}
    \ge  \mu > 0.
  \label{equ:infsup}
\end{equation}

The  finite element approximation of \eqref{variational} is to
find $u_h\in S_h$ satisfying
\begin{equation}
    B(u_h, v_h) = (f, v_h),  \quad \forall v_h \in S_h.
    \label{discrete}
\end{equation}
To insure a unique solution for \eqref{discrete}, we assume the inf-sup conditions
\begin{equation}
  \inf_{u\in S_h}\sup_{v\in S_h}
    \frac{B(u, v)}{\|u\|_{1, \Omega}\|v\|_{1, \Omega}}
    =
  \sup_{u\in S_h}\inf_{v\in S_h}
    \frac{B(u, v)}{\|u\|_{1, \Omega}\|v\|_{1, \Omega}}
    \ge  \mu > 0.
  \label{equ:disinfsup}
\end{equation}

From \eqref{variational} and \eqref{discrete}, it is easy to see that
\begin{equation}
  B(u-u_h, v) = 0
  \label{equ:orth}
\end{equation}
for any $v \in S_h$. In particular, \eqref{equ:orth} holds for any
$v \in S^{\text{comp}}_h(\Omega)$.

\subsection{Linear element}

Linear finite element space $S_h$ on quasi-uniform mesh $\mathcal{T}_h$
is considered in this subsection.

\begin{definition}
The triangulation $\mathcal{T}_h$ is said to satisfy Condition
$(\sigma,\alpha)$ if
there exist a partition $\mathcal{T}_{h,1} \cup \mathcal{T}_{h,2}$ of $\mathcal{T}_h$
and positive constants $\alpha$ and $\sigma$ such that every
two adjacent triangles in $\mathcal{T}_{h,1}$ form an $O(h^{1+\alpha})$ parallelogram and
$$
\sum_{T\in {\mathcal{T}_{h,2}}} |T| = O(h^\sigma).
$$
\end{definition}
An $O(h^{1+\alpha})$ parallelogram is a quadrilateral shifted from a parallelogram by $O(h^{1+\alpha})$.

For general $\alpha$ and $\sigma$, Xu and Zhang \cite{xu2004} proved the following theorem.

\begin{theorem}\label{linearsuperclose}
 Let $u$ be the solution of \eqref{variational}, let $u_h\in S_h$ be the finite element
solution of \eqref{discrete}, and let $u_I \in S_h$ be the linear interpolation of $u$. If the triangulation
$\mathcal{T}_h$ satisfies Condition $(\sigma,\alpha)$ and $u \in H^3(\Omega) \cap W^{2}_{\infty} (\Omega)$, then
$$
|u_h-u_I|_{1,\Omega} \lesssim  h^{1+\rho} (|u|_{3,\Omega} + |u|_{2,\infty,\Omega}),
$$
where $\rho = \min(\alpha,\sigma/2,1/2)$.
\end{theorem}

Using the above  result, we are able to obtain a convergence rate for our Hessian recovery operator.
\begin{theorem}
    Suppose that the solution of \eqref{variational} belongs to  $H^3(\Omega) \cap W^{2}_{\infty} (\Omega)$
     and $\mathcal{T}_h$ satisfies Condition $(\sigma,\alpha)$,
     then we have
    \begin{equation*}
	\|H u - H_h u_h\|_{0,\Omega}\le  h^\rho \|u\|_{3,\infty, \Omega}.
    \end{equation*}
    \label{linearsuperconvergent}
\end{theorem}
\begin{proof}
   We decompose $Hu-H_hu_h$ as $(Hu-H_hu)+H_h(u_I-u_h)$, since $H_hu=H_hu_I$.
   Using the triangle inequality and the definition of $H_h$, we obtain
   \begin{align*}
       	\|H u - H_h u_h\|_{0, \Omega}&\le
	\|Hu-H_hu\|_{0, \Omega}+\|H_h(u_I-u_h)\|_{0, \Omega}\\
	&=\|Hu-H_hu\|_{0, \Omega}+\|G_h(G_h(u_I-u_h))\|_{0, \Omega}.
   \end{align*}
   The first term in the above expression is bounded by
   $h|u|_{3, \infty, \Omega}$ according to Theorem \ref{pp}.
   Since $G_h$ is a bounded linear operator \cite{naga2005},
   it follows that
   \begin{equation*}
       \|H_h(u_I-u_h)\|_{0, \Omega}\lesssim \|\nabla(G_h(u_I-u_h))\|_{0, \Omega}
   \end{equation*}
   Notice that $G_h (u_I-u_h)$ is a function in $S_h$ and hence  the inverse
   estimate \cite{ciarlet1978, brenner2008} can be applied. Thus,
   \begin{align*}
       \|H_h(u_I-u_h)\|_{0, \Omega}\lesssim h^{-1}\|G_h(u_I-u_h)\|_{0, \Omega}
       \lesssim h^{-1}\|u_I-u_h\|_{1, \Omega}
   \end{align*}
   and hence Theorem \ref{linearsuperclose} implies that
   \begin{align*}
       \|H_h(u_I-u_h)\|_{0, \Omega}
       \lesssim  h^\rho \|u\|_{3,\infty,\Omega} .
   \end{align*}
   Combining the above two estimates completes our proof.
\end{proof}

%

\subsection{Quadratic element}
We proceed to quadratic finite element space $S_h$.
According to \cite{huang2008}, a triangulation $\mathcal{T}_h$ is strongly
regular if any two adjacent triangles in $\mathcal{T}_h$ form an $O(h^2)$
approximate parallelogram. Huang and Xu  proved the following
superconvergence results in \cite{huang2008}.
\begin{theorem}
    If the triangulation $\mathcal{T}_h$ is uniform or strongly regular,
    then
    \begin{equation*}
	|u_h-u_I|_{1, \Omega}\lesssim h^3|u|_{4,  \Omega}.
    \end{equation*}
    \label{quadsuperclose}
\end{theorem}
Based on the above theorem, we obtain the following superconvergent result.
\begin{theorem}
    Suppose that the solution of \eqref{variational} belongs to
    $H^{4}(\Omega)$ and $\mathcal{T}_h$ is uniform or strongly regular.
     Then we have
    \begin{equation*}
	\|H u - H_h u_h\|_{0, \Omega}\le h^{2}
	\|u\|_{4, \Omega}.
    \end{equation*}
    \label{quadsuperconvergent}
\end{theorem}
\begin{proof}
    The proof is similar to the proof of Theorem \ref{linearsuperconvergent}
    by using Theorem \ref{quadsuperclose} and  the inverse estimate.
\end{proof}
\begin{remark}
    Theorem \ref{quadsuperconvergent} can be generalized to mildly structured
    meshes as in \cite{huang2008}.
\end{remark}

\subsection{Translation invariant element of any order}
First, we observe that the Hessian recovery operator results in a difference
quotient.  It is due to the fact that $G_h$ is a difference quotient
\cite{zhang2005} and the composition of two difference quotients is still a
difference quotient. Let us take linear element on uniform triangular mesh of
the regular pattern as an example, see Figure \ref{regular}.
The recovered second order derivative  at a nodal point $z$ is
\begin{equation*}
\begin{split}
  (H_h^{xx}u_h)(z) = \frac{1}{36h^2}(&-12u_0+2u_1-4u_2-4u_3+2u_4-4u_5-4u_6
   +4u_7 + 4u_8 +u_9\\
   &-2u_{10}+u_{11}+4u_{12}+4u_{13}+
  4u_{14}+u_{15}-2u_{16}+u_{17}+4u_{18}).
\end{split}
\end{equation*}
Let $\phi_j$ be the nodal shape functions. Since $\phi_{z}(z') = \delta_{zz'}$,
it follows that
\begin{align*}
  &(H_h^{xx}u_h)\phi_0(x, y) \\
  = &\frac{1}{36h^2}[
  -12u_0\phi_0(x, y)+2u_1\phi_1(x+h, y)-4u_2\phi_2(x+h, y+h)\\
  &-4u_3\phi_3(x, y+h)
  +2u_4\phi_4(x-h, y)-4u_5\phi_5(x-h, y-h)\\
  &-4u_6\phi_6(x, y-h) +4u_7\phi_7(x+2h, y) + 4u_8\phi_8(x+2h, y+h)\\
  &+u_9\phi_9(x+2h, y+2h)-2u_{10}\phi_{10}(x+h, y+2h)+u_{11}\phi_{11}(x, y+2h)\\
  &+4u_{12}\phi_{12}(x-h, y+h)+4u_{13}\phi_{13}(x-2h, y)+
  4u_{14}\phi_{14}(x-2h, y-h)\\
  &+u_{15}\phi_{15}(x-2h, y-2h)-2u_{16}\phi_{16}(x-h, y-2h)+
  u_{17}\phi_{17}(x, y-2h)\\
&+4u_{18}\phi_{18}(x+h, y-h)].
\end{align*}
The translations are in the directions of $\ell_1=(1, 0)$,
$\ell_2=(0, 1)$, $\ell_3=(\frac{\sqrt{2}}{2}, \frac{\sqrt{2}}{2})$,
$\ell_4=(\frac{\sqrt{2}}{2}, -\frac{\sqrt{2}}{2})$,
$\ell_5=(\frac{\sqrt{5}}{5}, \frac{2\sqrt{5}}{5})$, and
$\ell_6=(\frac{2\sqrt{5}}{5}, \frac{\sqrt{5}}{5})$.
Therefore, we can express the recovered second order derivative as
\begin{equation}
  (H_h^{xx}u_h)(z) = \sum_{|\nu|\le M} \sum_{i=1}^{6}
  C^{i}_{\nu, h}u_h(z+\nu h\ell_i),
  \label{equ:hessian}
\end{equation}
for some integer $M$.

Let all coefficients in the bilinear form  $B(\cdot, \cdot)$ be constant. Then
\begin{equation*}
  B(T^{\ell}_{\nu\tau}(u-u_h), v) = B(u-u_h, T^{\ell}_{-\nu\tau}v)
  =B(u-u_h, (T^{\ell}_{\nu\tau})^{*}v) = 0.
\end{equation*}
Since $H_{h}^{xx}$ is a difference operator constructed from translation of
type \eqref{equ:hessian}, it follows that
\begin{equation}
  B(H_h^{xx}(u-u_h), v) = B(u-u_h, (H^{xx}_h)^{*}v) = 0,
  \quad v\in S^{\text{comp}}_h(\Omega_1).
  \label{equ:test}
\end{equation}
Therefore,  Theorem 5.5.2 of \cite{wahlbin1995} (with $F\equiv 0$)
implies that
\begin{equation}
  \begin{split}
  \|H_{h}^{xx}(u-u_h)\|_{0, \infty, \Omega_0}
  \lesssim&
  \left(\ln\frac{d}{h}\right)^{\bar{r}}\min_{v\in S_h}
  \|H_h^{xx}u-v\|_{0, \infty, \Omega_1}\\
  &+ d^{-s-\frac{2}{q}}
  \|H_{h}^{xx}(u-u_h)\|_{-s, q, \Omega_1}.
  \label{equ:result}
  \end{split}
\end{equation}
Here $\bar{r} = 1$ for linear element and $\bar{r}=0$ for higher order element.
Note that $H_h^{xx}u\in S_h$ and hence the first term on the right hand
side of \eqref{equ:result} can be estimated by standard approximation theory
under the assumption that the finite element space includes piecewise
polynomial of degree $k$:
\begin{equation}
  \min_{v\in S_h}
  \|H_h^{xx}u-v\|_{0, \infty, \Omega_1}
  \lesssim h^{k+1}|u|_{k+3, \infty, \Omega_1},
  \label{equ:first}
\end{equation}
provided $u\in W^{k+3}_{\infty}(\Omega)$, see \cite{brenner2008, ciarlet1978}.
It remains to attack the second term on the right hand side of
\eqref{equ:result}. Note that
\begin{equation}
  \|H_{h}^{xx}(u-u_h)\|_{-s, q, \Omega_1}
  =\sup_{\phi\in C^{\infty}_0(\Omega_1), \|\phi\|_{s, q', \Omega_1}=1}
  (H^{xx}_h(u-u_h), \phi).
  \label{equ:second}
\end{equation}
Here $\frac{1}{q}+\frac{1}{q'} = 1$ and
\begin{equation}
  \begin{split}
    (H^{xx}_h(u-u_h), \phi)& = (u-u_h, (H^{xx}_h)^*\phi)\\
    &\lesssim \|u-u_h\|_{0, \infty, \Omega_2}
    \|(H^{xx}_h)^*\phi\|_{0, 1, \Omega_2}\\
    &\lesssim \|u-u_h\|_{0, \infty, \Omega_2},
  \end{split}
  \label{equ:mid}
\end{equation}
where we use the fact that $\|(H^{xx}_h)^*\phi\|_{0, 1, \Omega_2}$ is bounded
uniformly with respect to $h$ when $s\ge 1$.  We now once again apply
Theorem 5.5.1 from \cite{wahlbin1995} to $ \|u-u_h\|_{0, \infty, \Omega_2}$
with $\Omega_2\subset\subset \Omega$ separated by $d$, then
\begin{equation}
  \begin{split}
    \|u-u_h\|_{0, \infty, \Omega_2} \lesssim&
    \left(\ln\frac{d}{h}\right)^{\bar{r}}\min_{v\in S_h}
  \|u-v\|_{0, \infty, \Omega}\\
  &+ d^{-s-\frac{2}{q}}
  \|u-u_h\|_{-s, q, \Omega}.
  \end{split}
  \label{equ:temp}
\end{equation}

If the separation parameter $d = O(1)$ , then we combine
\eqref{equ:result}, \eqref{equ:first} and \eqref{equ:temp} to obtain
\begin{equation}
  \|H_{h}^{xx}(u-u_h)\|_{0, \infty, \Omega_0}
  \lesssim \left(\ln\frac{1}{h}\right)^{\bar{r}}
  h^{k+1}\|u\|_{k+3, \infty, \Omega}
  + \|u-u_h\|_{-s, q, \Omega}.
  \label{equ:any}
\end{equation}
Following the same argument, we can establish the same result for
$H_h^{xy}$, $H_h^{yx}$, and $H_h^{yy}$. Therefore, \eqref{equ:any} is
satisfied by replacing $H_h^{xx}$ with $H_h$:
\begin{equation}
  \|H_{h}(u-u_h)\|_{0, \infty, \Omega_0}
  \lesssim \left(\ln\frac{1}{h}\right)^{\bar{r}}
  h^{k+1}\|u\|_{k+3, \infty, \Omega}
  + \|u-u_h\|_{-s, q, \Omega}.
  \label{equ:error}
\end{equation}

Now we are in a perfect position to prove our main result for
translation invariant finite element space of any order.
\begin{theorem}
  Let all the coefficients in the bilinear operator $B(\cdot, \cdot)$
  be constant;  let $\Omega_0\subset\subset \Omega_2 \subset\subset \Omega$
  be separated by $d=O(1)$; let the finite element space $S_h$, which includes
  piecewise polynomials of degree $k$, be translation invariant in the
  directions required by the Hessian recovery operator $H_h$ on
  $\Omega_2$; and let $u\in W^{k+3}_{\infty}(\Omega)$.
  Assume that Theorem 5.2.2 from \cite{wahlbin1995} is applicable. Then
  \begin{equation}
   \|Hu - H_{h}u_h\|_{0, \infty, \Omega_0}
  \lesssim \left(\ln\frac{1}{h}\right)^{\bar{r}}
  h^{k+1}\|u\|_{k+3, \infty, \Omega}
  + \|u-u_h\|_{-s, q, \Omega}.
    \label{equ:main}
  \end{equation}
  for some $s\ge 0$ and $q\ge 1$.
  \label{thm:ultra}
\end{theorem}
\begin{proof}
  We decompose
  \begin{equation}
    Hu - H_h u_h = (Hu - (Hu)_I) + ((Hu)_I - H_hu) + H_h(u-u_h),
    \label{equ:decomp}
  \end{equation}
  where $(Hu)_I \in S_h^2\times S_h^2$ is the standard Lagrange interpolation
  of $Hu$ in the finite element space $S_h$. By the standard
  approximation theory, we obtain
  \begin{equation}
   \|Hu - (Hu)_I\|_{0, \infty, \Omega}
   \lesssim h^{k+1} |Hu|_{k+1, \infty, \Omega}
   \lesssim h^{k+1} |u|_{k+3, \infty, \Omega}.
    \label{equ:mainfirst}
  \end{equation}
  For the second term, using Theorem \ref{pp},  we have
  \begin{equation}
    \begin{split}
    \|(Hu)_I - H_h u\|_{0, \infty, \Omega_0}=&
    \|\sum_{z \in \mathcal{N}_h}((Hu)(z) - (H_hu)(z))\phi_{z}\|
    _{0, \infty, \Omega_0}\\
    \lesssim&
    \max_{z\in \mathcal{N}_h \cap \Omega_0} |(Hu)(z) - (H_hu)(z)|\\
    \lesssim&
    h^{k+1}|u|_{k+3, \infty, \Omega}.
    \end{split}
    \label{equ:mainsecond}
  \end{equation}
  The last term in \eqref{equ:decomp} is bounded by \eqref{equ:error}.
  The conclusion follows by combining \eqref{equ:error},
  \eqref{equ:mainfirst} and \eqref{equ:mainsecond}.
\end{proof}

\begin{remark}
 Theorem \ref{thm:ultra} is a ultraconvergence result under the condition
 \begin{equation*}
   \|u-u_h\|_{-s, q, \Omega} \lesssim h^{k+\sigma}, \quad \sigma > 0.
 \end{equation*}
 The reader is referred to \cite{nitsche} for negative norm estimates.
\end{remark}
\section{Numerical tests}
In this section, two numerical examples are  provided to illustrate
 our Hessian recovery method.
 The first one is designed to
demonstrate the polynomial preserving property of the proposed Hessian recovery
method. The second one is devoted  to
a comparison of our method and some existing Hessian recovery methods
in the literature on  both uniform and unstructured meshes.

In order to evaluate the performance of Hessian recovery methods,
we split mesh nodes
$\mathcal{N}_h$ into $\mathcal{N}_{h,1}$ and $\mathcal{N}_{h,2}$, where
$\mathcal{N}_{h,2}=\{z\in \mathcal{N}_h: \text{dist}(z,\partial\Omega)\leq L\}$
denotes the set of nodes near boundary and
$\mathcal{N}_{h,1} = \mathcal{N}_h\setminus \mathcal{N}_{h, 2}$
denotes rest interior nodes.
 Now, we can define
$$\Omega_{h, 1}=\bigcup\{\tau\in \mathcal{T}_h: \tau\ \text{has\ all\
of\ its\ vertices\ in}\ \mathcal{N}_{h,1}\}, $$ and
$\Omega_{h, 2}=\Omega \setminus \Omega_{h, 1} $.
In the following examples  we choose  $L=0.1$.

Let $\tilde{G}_h^{}$  be the weighted average recovery operator.
Then we define
\begin{equation*}
    H^{ZZ}_hu_h=
    \begin{pmatrix} \tilde{G}_h(\tilde{G}_h^xu_h),&\tilde{G}_h(\tilde{G}_h^yu_h)
    \end{pmatrix},
\end{equation*}
and
\begin{equation*}
    H^{LS}_hu_h=
    \begin{pmatrix} \tilde{G}_h(G_h^xu_h),&\tilde{G}_h(G_h^yu_h)
    \end{pmatrix}.
\end{equation*}
For any nodal point $z$, fit a quadratic polynomial $p_z$ at $z$ as PPR.
Then $H^{QF}_h$ is defined as
\begin{equation*}
    H^{QF}_hu_h(z)=
    \begin{pmatrix} \frac{\partial^2 p_{z}}{\partial x^2}(0, 0)&
            \frac{\partial^2 p_{z}}{\partial x\partial y}(0, 0)\\
            \frac{\partial^2 p_{z}}{\partial y\partial x}(0, 0)&
	    \frac{\partial^2 p_{z}}{\partial y^2}(0, 0)
    \end{pmatrix}.
\end{equation*}
$H^{ZZ}_h$, $H^{LS}_h$, and $H^{QF}_h$ are the first three Hessian recovery
methods in \cite{picasso2011}.  To compare them, define
    \begin{align*}
      &De = \|H_h u_h - Hu\|_{L^2(\Omega_{1, h})},\quad
    De^{ZZ} = \|H^{ZZ}_h u_h - Hu\|_{L^2(\Omega_{1, h})},\\
    &De^{LS} = \|H^{LS}_h u_h - Hu\|_{L^2(\Omega_{1, h})},\quad
    De^{QF} = \|H^{QF}_h u_h - Hu\|_{L^2(\Omega_{1, h})}.
    \end{align*}
where $u_h$ is the finite element solution.

{\bf Example 1.} Consider the following function
\begin{equation}
    u(x, y) = sin(\pi x)sin(\pi y),
    \quad (x, y) \in \Omega = (0, 1)\times (0, 1).
    \label{cubic}
\end{equation}
Let $u_I$ be the standard Lagrangian interpolation of $u$ in the finite element
space.  To validate Theorem \ref{pp}, we apply the Hessian recovery operator
$H_h$ to $u_I$ and consider the discrete maximum error of $H_hu_I - Hu$ at all
vertices in $N_{1, h}$. First, linear element on uniform meshes
are taken into account.
Figures \ref{ex1r} -\ref{ex1u} display the numerical results.
The numerical errors decrease at a rate of $O(h^{2})$ for
 four different pattern uniform meshes.  It means the proposed Hessian
recovery method preserves polynomial of degree $3$ for linear element on
uniform meshes.

 Next, we consider unstructured meshes. We start from an initial mesh
generated by EasyMesh\cite{easymesh} as shown in Figure \ref{delaunay},
followed by four levels of refinement using bisection.
Figure \ref{ex1d} shows that the recovered Hessian $H_hu_I$ converges to the exact Hessian
at rate $O(h)$. This coincides with the result in Theorem \ref{poly}
that $H_h$ only preserves polynomials of  degree 2 on general unstructured meshes
\begin{figure}[!h]
  \centering
  \begin{minipage}[c]{0.5\textwidth}
  \centering
  \includegraphics[width=0.9\textwidth]{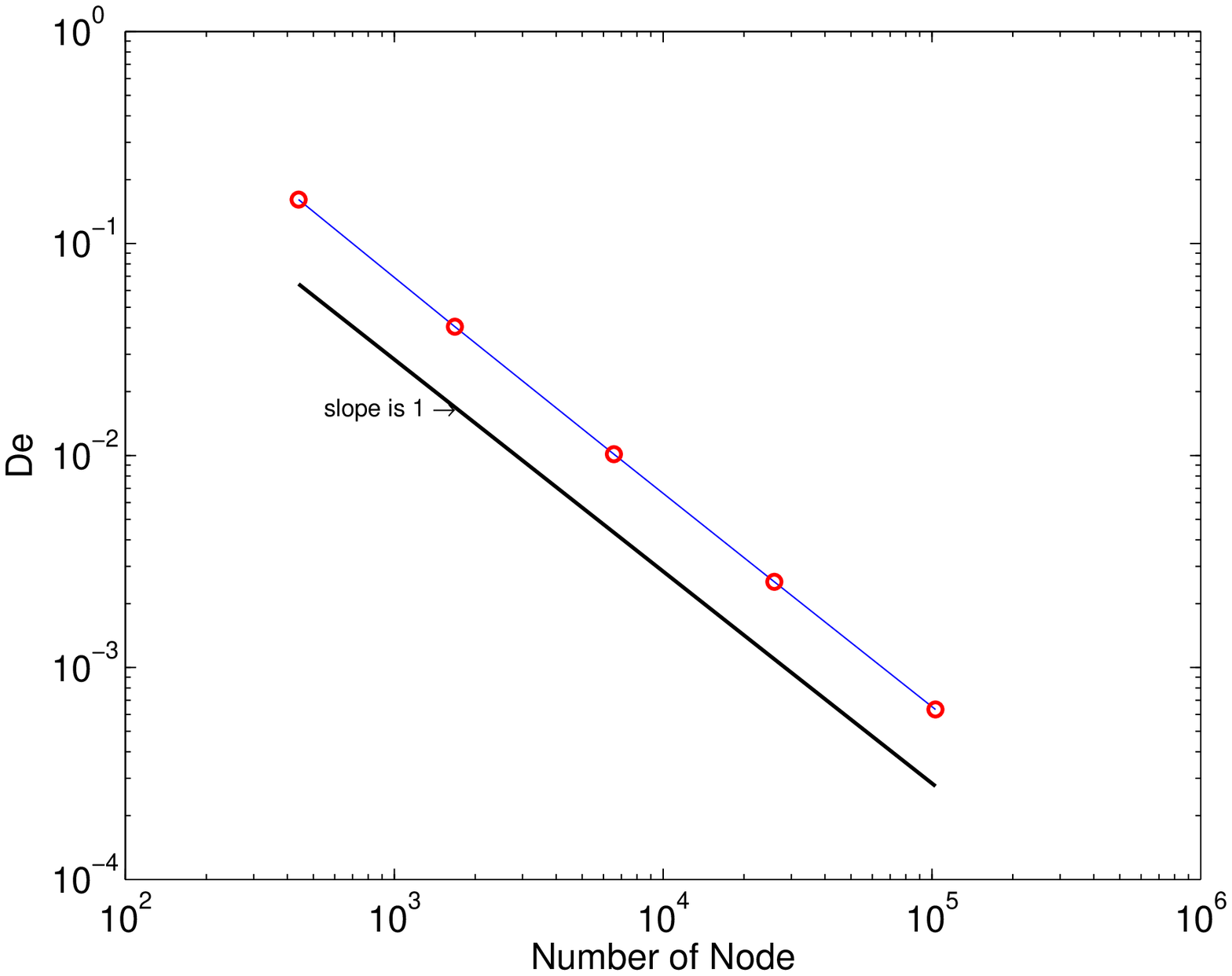}
  \caption{Linear Element: Regular Pattern}
\label{ex1r}
\end{minipage}%
 \begin{minipage}[c]{0.5\textwidth}
  \centering
  \includegraphics[width=0.9\textwidth]{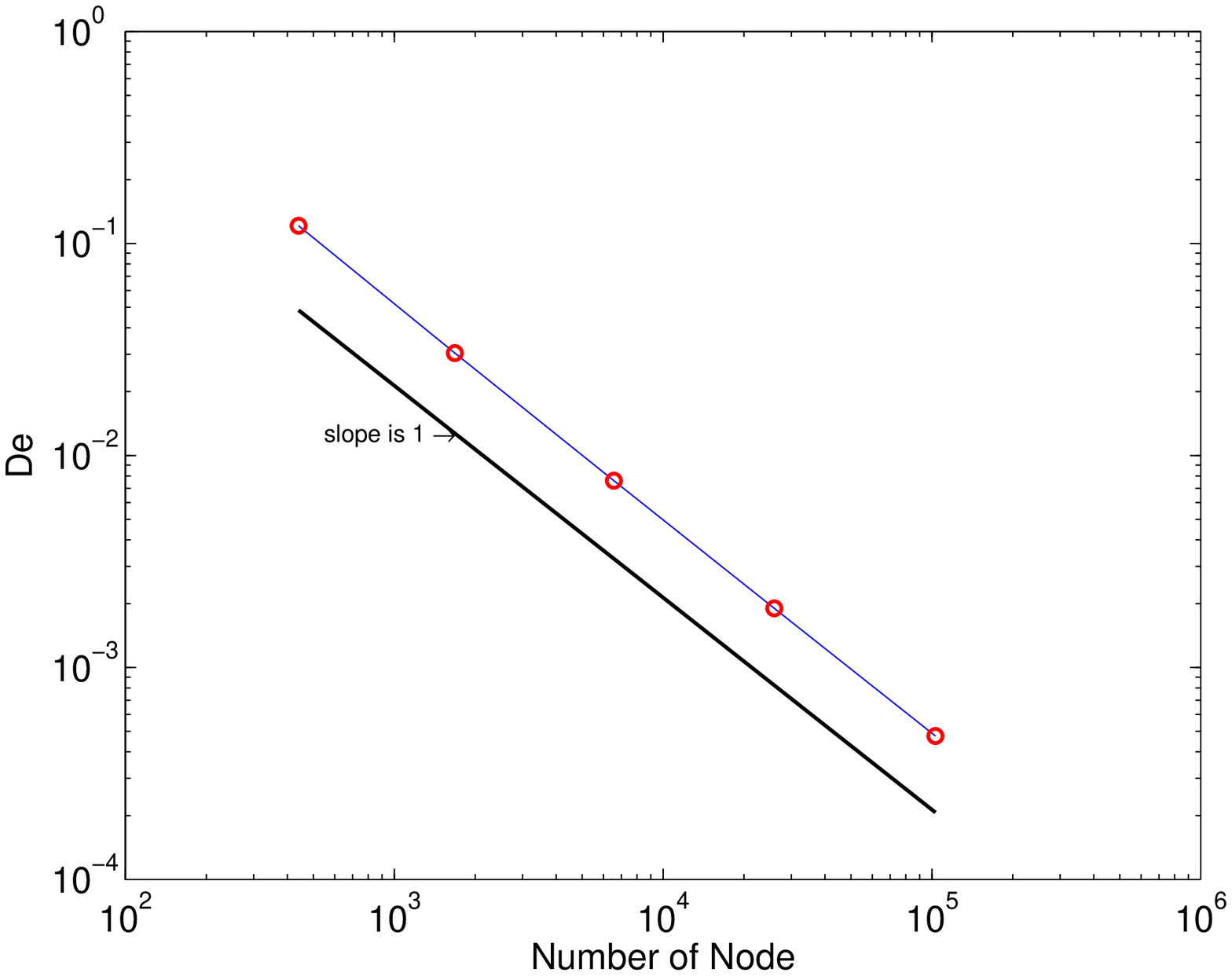}
   \caption{Linear Element: Chevron Pattern}
\label{ex1h}
\end{minipage}
\end{figure}

\begin{figure}[!h]
  \centering
  \begin{minipage}[c]{0.5\textwidth}
  \centering
  \includegraphics[width=0.9\textwidth]{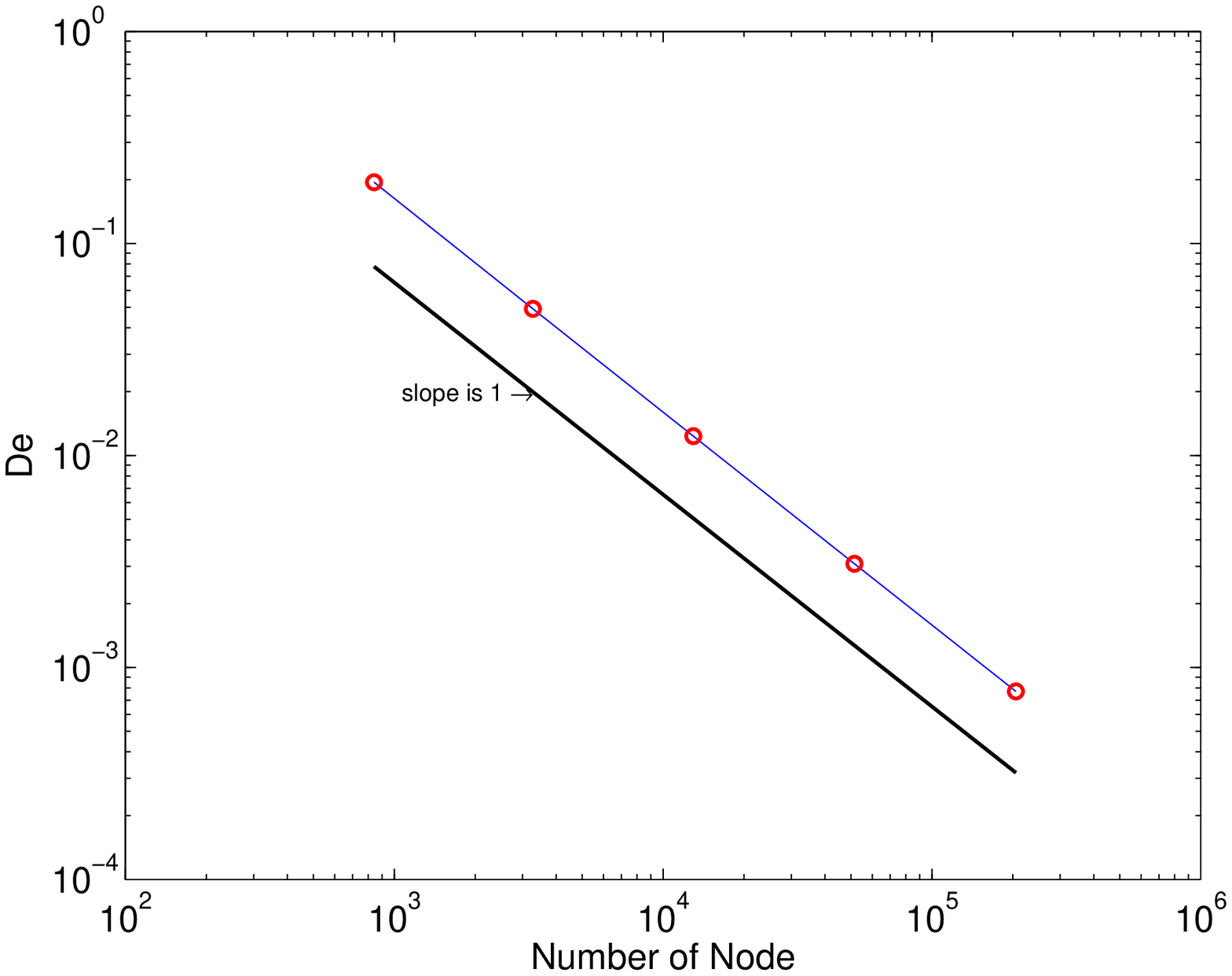}
  \caption{Linear Element:  Criss-cross Pattern}
\label{ex1c}
\end{minipage}%
 \begin{minipage}[c]{0.5\textwidth}
  \centering
  \includegraphics[width=0.9\textwidth]{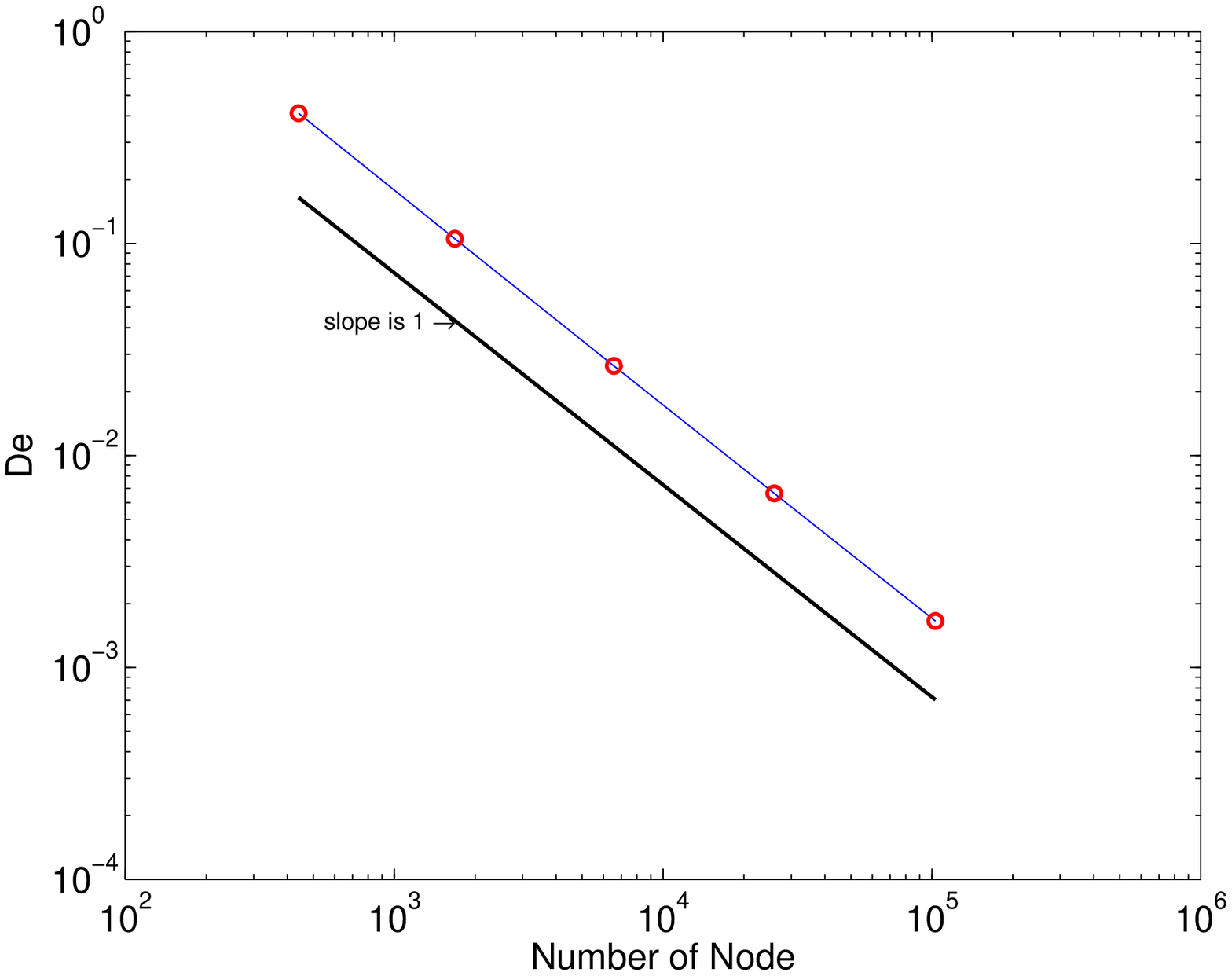}
   \caption{Linear Element:  Union-Jack Pattern}
\label{ex1u}
\end{minipage}
\end{figure}

\begin{figure}[!h]
  \centering
  \begin{minipage}[c]{0.5\textwidth}
  \centering
  \includegraphics[width=0.9\textwidth]{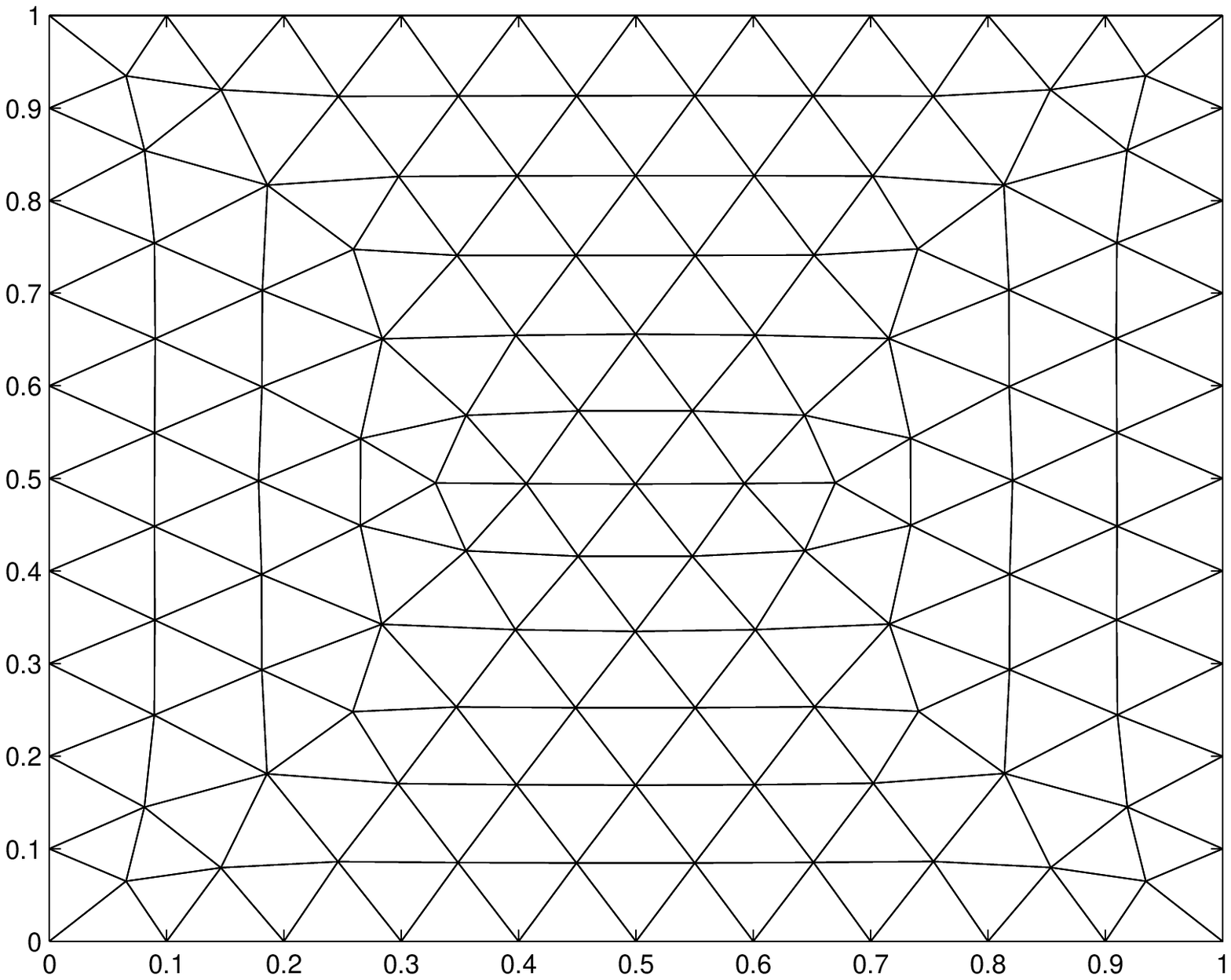}
  \caption{Delaunay}
\label{delaunay}
\end{minipage}%
 \begin{minipage}[c]{0.5\textwidth}
  \centering
  \includegraphics[width=0.9\textwidth]{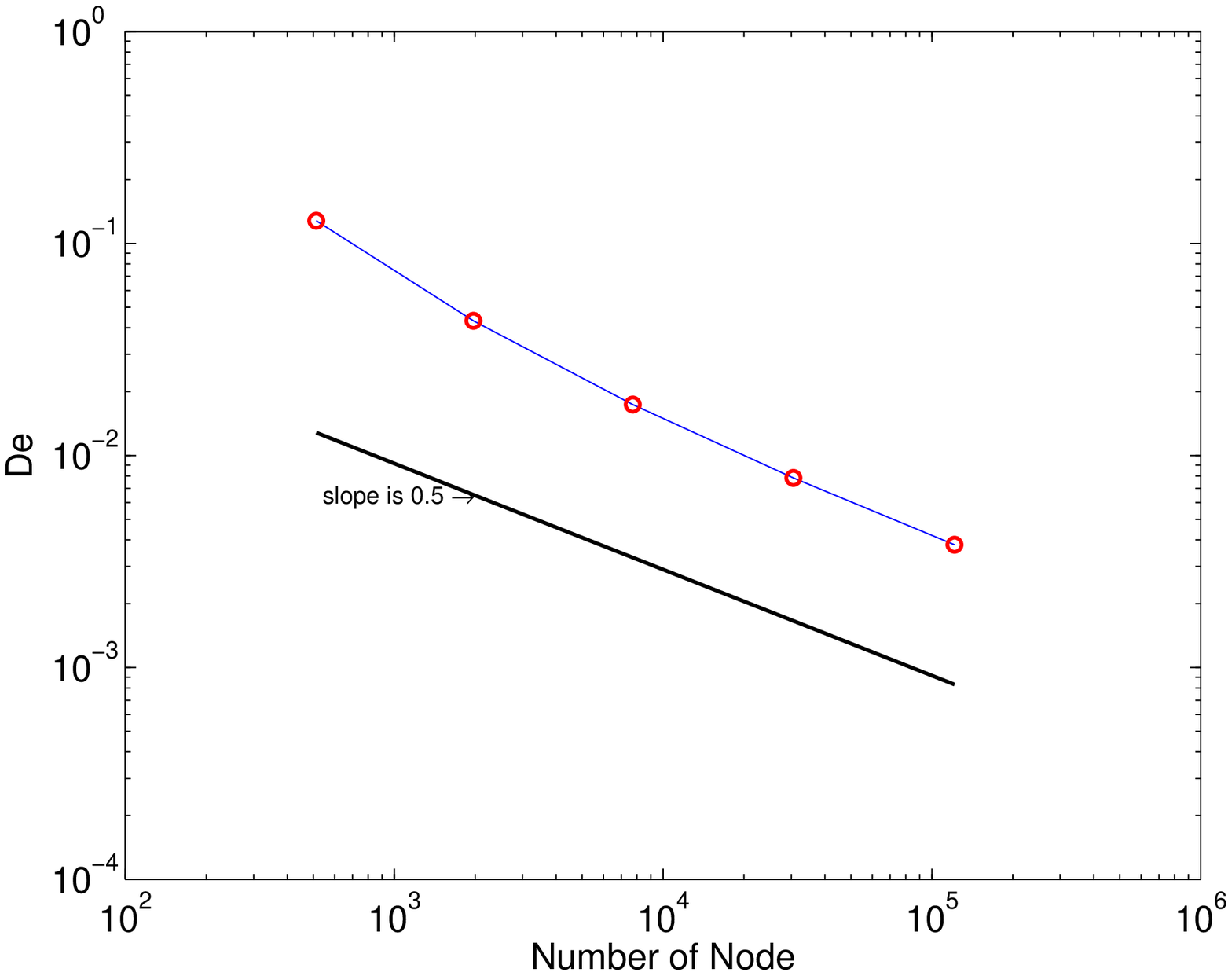}
   \caption{Linear Element: Delaunay Mesh}
\label{ex1d}
\end{minipage}
\end{figure}
Then we turn to quadratic element.
We test the discrete error of recovered Hessian $H_hu_I$ and the exact Hessian
$Hu$ using uniform meshes of regular pattern and the same Delaunay meshes.
Similarly, we define $\|\cdot\|_{\infty, h}$ as a discrete  maximum norm
at all vertices and edge centers in an interior region  $\Omega_{1, h}$.
The result of uniform mesh of regular pattern is reported
in Figure \ref{ex2r}.
As predicted by Theorem \ref{pp}, $H_hu_I$ converges to $Hu$ at rate of
$O(h^4)$ which implies $H_h$ preserves polynomials of degree $5$ for
quadratic element on uniform triangulation.
For unstructured mesh, we observe that
$H_hu_I$ approximates $Hu$ at a rate of $O(h^2)$ from Figure \ref{ex2d}.

\begin{figure}[!h]
  \centering
  \begin{minipage}[c]{0.5\textwidth}
  \centering
  \includegraphics[width=0.9\textwidth]{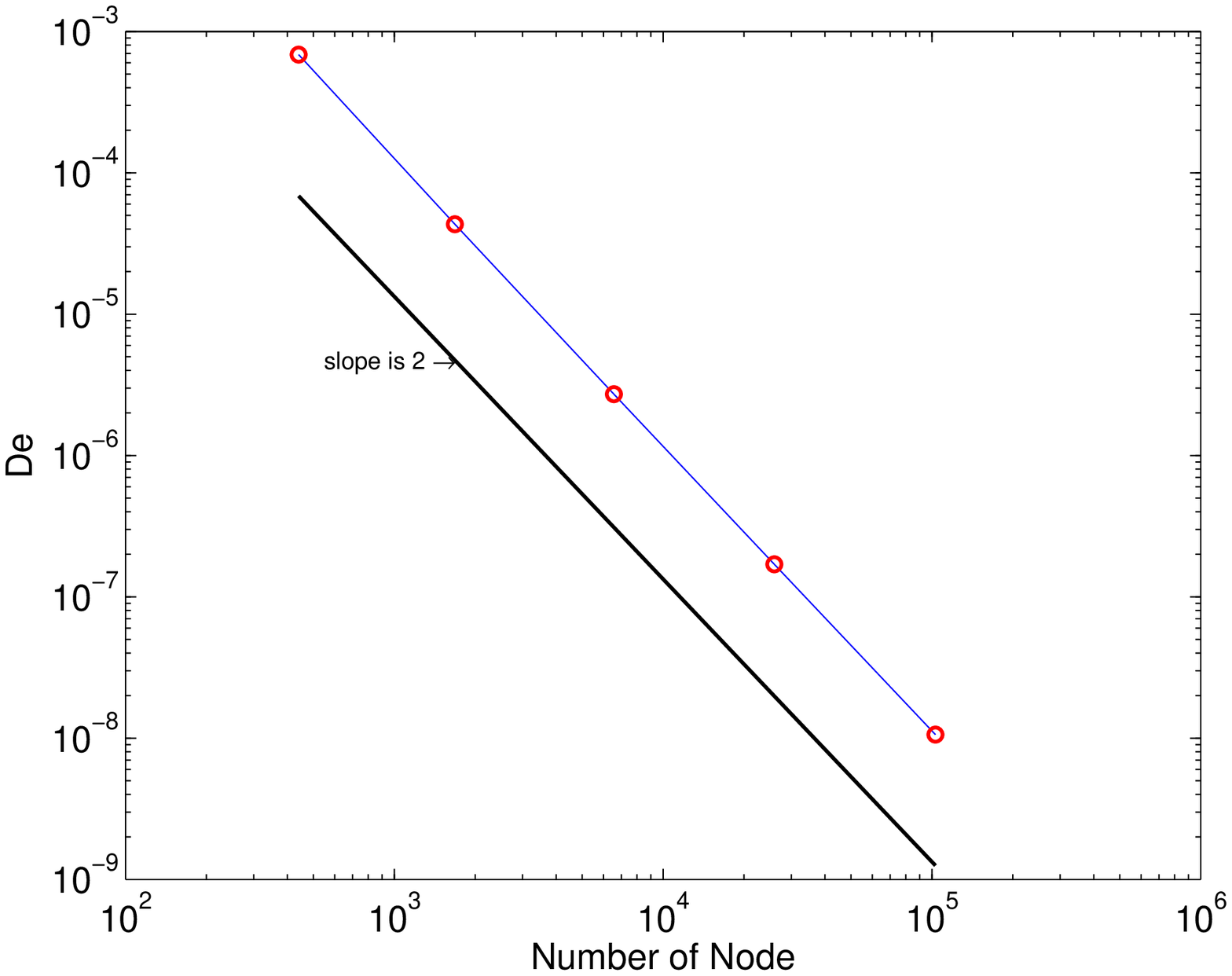}
  \caption{Quadratic Element: Regular Pattern}
\label{ex2r}
\end{minipage}%
 \begin{minipage}[c]{0.5\textwidth}
  \centering
  \includegraphics[width=0.9\textwidth]{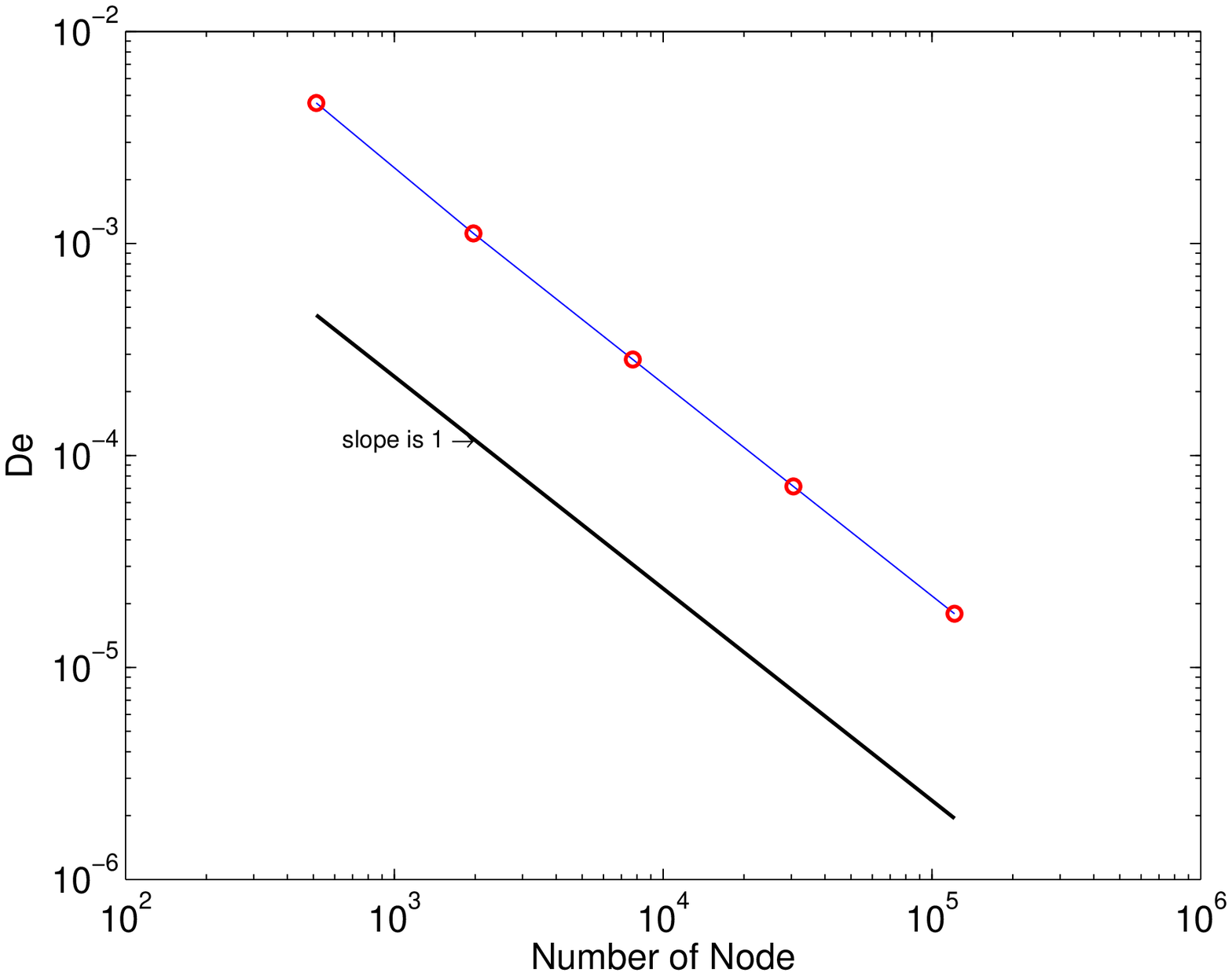}
   \caption{Quadratic Element: Delaunay Mesh}
\label{ex2d}
\end{minipage}
\end{figure}

{\bf Example 2.}  We consider  the following elliptic equation
\begin{equation}
    \begin{cases}
	-\Delta u =2\pi^2\sin\pi x\sin\pi y,&\text{in }
	\Omega=[0, 1] \times [0, 1],\\
	u = 0,&  \text{on } \partial \Omega.
    \end{cases}
    \label{ex3}
\end{equation}
The exact solution is $u(x, y) = \sin(\pi x)\sin(\pi y)$.
First, linear element is considered.
In Table \ref{ex3r}, we report the numerical results for regular pattern
meshes.   All four methods  ultraconverge at a rate of
$O(h^{2})$ in the interior subdomain.
The fact that $H_h^{LS}$ and $H^{ZZ}_h$
perform as good as $H_h$ is not a surprise since it is well known that the
polynomial preserving recovery is the same as weighted average for
uniform triangular mesh of the regular pattern.

The results of  the Chevron pattern is shown in Table \ref{ex3h}.
$H_hu_h$ approximates $Hu$ at  rate $O(h^{2})$ while
$H_h^{LS}u_h$, $H_h^{ZZ}u_h$ and $H_h^{QF}u_h$  approximate $Hu$ at rate $O(h)$.
It is observed that our method out-performs other three Hessian recovery
methods on the Chevron pattern uniform meshes.
To the best of our knowledge, the proposed  PPR-PPR Hessian recovery
is the only  method to achieve $O(h^2)$ superconvergence for linear
element under the Chevron pattern triangular mesh.

\begin{table}[htb!]
\centering
\small
\caption{Example 2: Regular Pattern }
\begin{tabular}{|c|c|c|c|c|c|c|c|c|}
\hline
Dof & $De$ & order& $De^{ZZ}e$ & order& $De^{LS}$ & order& $De^{QF}$ &  order\\ \hline
121&7.93e-001&--&9.73e-001&--&7.93e-001&--&4.01e-001&--\\ \hline
441&2.02e-001&1.06&2.02e-001&1.22&2.02e-001&1.06&1.03e-001&1.05\\ \hline
1681&5.10e-002&1.03&5.10e-002&1.03&5.10e-002&1.03&2.61e-002&1.03\\ \hline
6561&1.28e-002&1.02&1.28e-002&1.02&1.28e-002&1.02&6.53e-003&1.02\\ \hline
25921&3.20e-003&1.01&3.20e-003&1.01&3.20e-003&1.01&1.63e-003&1.01\\ \hline
103041&8.00e-004&1.00&8.00e-004&1.00&8.00e-004&1.00&4.08e-004&1.00\\ \hline
\end{tabular}
\label{ex3r}
\end{table}

\begin{table}[htb!]
\centering
\small
\caption{Example 2: Chevron Pattern }
\begin{tabular}{|c|c|c|c|c|c|c|c|c|}
\hline
Dof & $De$ & order& $De^{ZZ}e$ & order& $De^{LS}$ & order& $De^{QF}$ &  order\\ \hline
121&6.51e-001&--&7.98e-001&--&7.82e-001&--&9.03e-001&--\\ \hline
441&1.34e-001&1.22&2.12e-001&1.03&2.34e-001&0.93&4.30e-001&0.57\\ \hline
1681&3.38e-002&1.03&7.96e-002&0.73&9.87e-002&0.64&2.11e-001&0.53\\ \hline
6561&8.46e-003&1.02&3.57e-002&0.59&4.68e-002&0.55&1.05e-001&0.51\\ \hline
25921&2.11e-003&1.01&1.73e-002&0.53&2.30e-002&0.52&5.23e-002&0.51\\ \hline
103041&5.29e-004&1.00&8.57e-003&0.51&1.15e-002&0.50&2.62e-002&0.50\\ \hline
\end{tabular}
\label{ex3h}
\end{table}

\begin{table}[htb!]
\centering
\small
\caption{Example 2: Criss-cross }
\begin{tabular}{|c|c|c|c|c|c|c|c|c|}
\hline
Dof & $De$ & order& $De^{ZZ}e$ & order& $De^{LS}$ & order& $De^{QF}$ &  order\\ \hline
221&5.49e-001&--&3.57e-001&--&4.40e-001&--&7.14e-001&--\\ \hline
841&1.28e-001&1.09&8.03e-002&1.12&1.04e-001&1.08&6.17e-001&0.11\\ \hline
3281&3.22e-002&1.01&2.01e-002&1.02&2.62e-002&1.01&5.95e-001&0.03\\ \hline
12961&8.06e-003&1.01&5.04e-003&1.01&6.55e-003&1.01&5.90e-001&0.01\\ \hline
51521&2.02e-003&1.00&1.26e-003&1.00&1.64e-003&1.00&5.89e-001&0.00\\ \hline
205441&5.04e-004&1.00&3.15e-004&1.00&4.09e-004&1.00&5.88e-001&0.00\\ \hline
\end{tabular}
\label{ex3c}
\end{table}

\begin{table}[htb!]
\centering
\small
\caption{Example 2: Unionjack Pattern }
\begin{tabular}{|c|c|c|c|c|c|c|c|c|}
\hline
Dof & $De$ & order& $De^{ZZ}e$ & order& $De^{LS}$ & order& $De^{QF}$ &  order\\ \hline
121&1.25e+000&--&8.40e-001&--&9.87e-001&--&1.05e+000&--\\ \hline
441&3.16e-001&1.06&1.77e-001&1.20&2.48e-001&1.07&6.95e-001&0.32\\ \hline
1681&7.96e-002&1.03&4.46e-002&1.03&6.24e-002&1.03&6.14e-001&0.09\\ \hline
6561&2.00e-002&1.02&1.12e-002&1.02&1.56e-002&1.02&5.95e-001&0.02\\ \hline
25921&5.00e-003&1.01&2.80e-003&1.01&3.91e-003&1.01&5.90e-001&0.01\\ \hline
103041&1.25e-003&1.00&6.99e-004&1.00&9.78e-004&1.00&5.89e-001&0.00\\ \hline
\end{tabular}
\label{ex3u}
\end{table}

\begin{table}[htb!]
\centering
\small
\caption{Example 2: Delaunay Mesh }
\begin{tabular}{|c|c|c|c|c|c|c|c|c|}
\hline
Dof & $De$ & order& $De^{ZZ}e$ & order& $De^{LS}$ & order& $De^{QF}$ &  order\\ \hline
139&4.31e-001&--&4.38e-001&--&4.40e-001&--&3.26e-001&--\\ \hline
513&1.38e-001&0.87&2.20e-001&0.53&1.49e-001&0.83&1.79e-001&0.46\\ \hline
1969&5.39e-002&0.70&2.36e-001&-0.05&5.85e-002&0.69&8.88e-002&0.52\\ \hline
7713&2.38e-002&0.60&1.62e-001&0.28&2.55e-002&0.61&4.35e-002&0.52\\ \hline
30529&1.14e-002&0.54&1.13e-001&0.26&1.19e-002&0.56&2.15e-002&0.51\\ \hline
121473&5.59e-003&0.51&7.97e-002&0.25&5.73e-003&0.53&1.07e-002&0.51\\ \hline
\end{tabular}
\label{ex3d}
\end{table}

Then  the Criss-cross pattern mesh is considered and  results are displayed
in Table \ref{ex3c}.  An $O(h^{2})$ convergence rate is observed for
our recovery method,  $H^{LS}_h$ and $H^{ZZ}_h$ while
no convergence rate is observed for  $H^{QF}_h$.
The results for  the Union-Jack pattern mesh is very
similar to  the Criss-cross pattern mesh except that our recovery method
superconverges at  rate $O(h^2)$ as shown in  Table \ref{ex3u}.

Now, we turn to unstructured mesh generated by EasyMesh \cite{easymesh} as
in  the previous examples.  Numerical data are listed in Table \ref{ex3d}.
$H_h$, $H^{LS}_h$ and $H^{QF}_h$ converge at a rate of $O(h^2)$ while
$H_h^{ZZ}$ only converges at a rate of $O(h)$.

 The  results above indicate clearly that our Hessian recovery
method converges at rate $O(h)$ on general Delaunay meshes, which
is  predicted by Theorem \ref{linearsuperconvergent}.
On uniform meshes, we can obtain $O(h^{2})$ ultraconvergence
on an interior sub-domain as predicted by Theorem  \ref{thm:ultra}.

\begin{figure}[!h]
  \centering
  \begin{minipage}[c]{0.5\textwidth}
  \centering
  \includegraphics[width=0.9\textwidth]{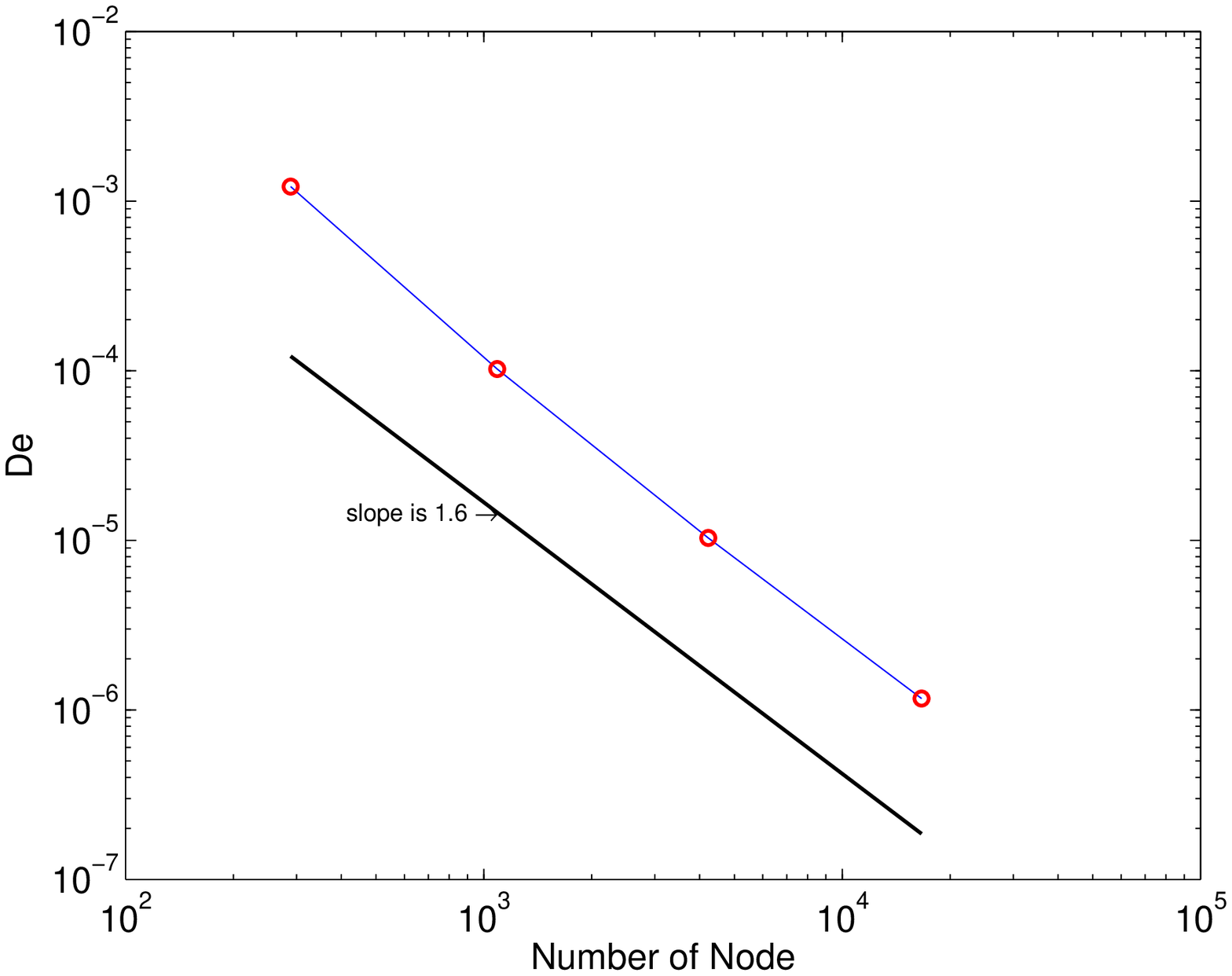}
  \caption{Example 2: Quadratic Regular Pattern}
\label{fig:ex3r}
\end{minipage}%
 \begin{minipage}[c]{0.5\textwidth}
  \centering
  \includegraphics[width=0.9\textwidth]{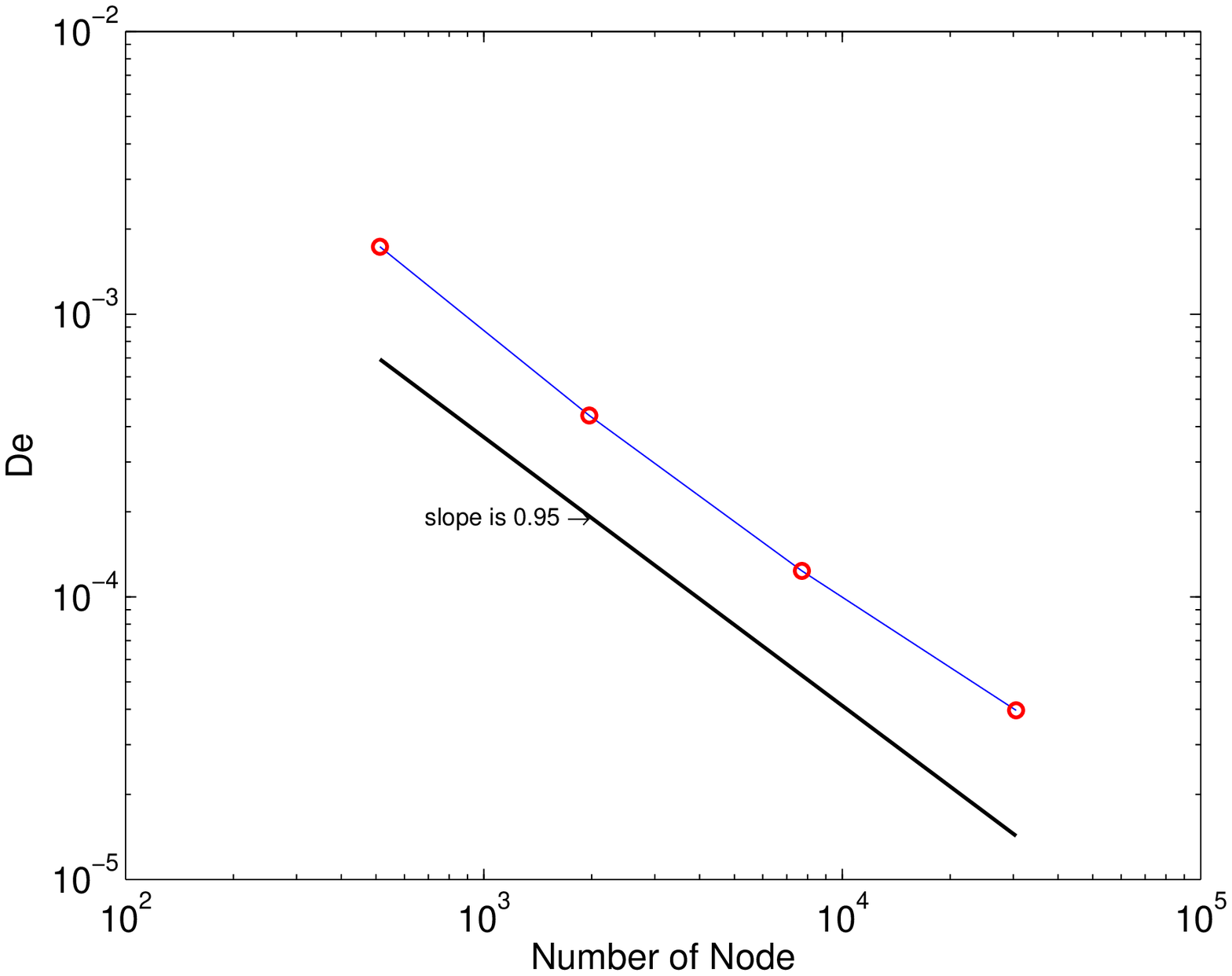}
   \caption{Example 2: Quadratic Delaunay Mesh}
\label{fig:ex3d}
\end{minipage}
\end{figure}

In the end,  we consider quadratic element. Note that our Hessian recovery method
is well defined for arbitrary order elements. However, the extension of
the other  three methods to quadratic element is not
straightforward or even impossible and hence only our method  is implemented here.
We report the numerical results in Figure \ref{fig:ex3r}
for regular pattern uniform mesh.
About $O(h^{3.2})$ order convergence
is observed, which is a bit better than the theoretical result predicted by
Theorem \ref{thm:ultra}.  Figure \ref{fig:ex3d} shows the result
for Delaunay mesh generated by EasyMesh \cite{easymesh}.
About $O(h^{1.9}$ superconvergence is observed.

\section{Concluding remarks}
In this  work, we introduced a  Hessian recovery method for
arbitrary order Lagrange  finite elements. Theoretically, we proved that
 the PPR-PPR Hessian recovery operator $H_h$ preserves polynomials
of degree $k+1$ on general unstructured
 meshes and preserves  polynomials of degree $k+2$ on  translation
invariant  meshes.  This polynomial preserving property, combined with the supercloseness property of
the finite element method, enabled us to prove convergence and superconvergence results for
our Hessian recovery method on mildly structured meshes.
Moreover, we proved the ultraconvergence result for translation invariant
finite element space of any order by using the argument of superconvergence
by difference quotient from \cite{wahlbin1995}.


\begin{thebibliography}{10}


 \bibitem{agouzal2002}
     {\sc A. Agouzal and Yu. Vassilevski},
     {\em On a discrete Hessian recovery for $P_1$ finite elements},
     J. Numer. Math., 10(2002), 1--12.
 \bibitem{bank2003}
     {\sc R. E. Bank and J. Xu},
     {\em Asymptotically exact a posteriori error estimators.
     I. Grids with superconvergence}, SIAM J. Numer. Anal. 41(2003),
     2294--2312.
 \bibitem{brenner2008}
     {\sc S.C. Brenner and L.R.  Scott},
     {\em The mathematical theory of finite element methods}, Third edition,
	 Texts in Applied Mathematics, 15. Springer, New York,  2008.
  \bibitem{cao2013}
      {\sc W. Cao},
      {\em Superconvergence analysis of the linear finite element method
      and a gradient recovery post-processing on anisotropic mesh},
	  Math. Comp.(2013) in press.
 \bibitem{ciarlet1978}
     {\sc P.G. Ciarlet},
     {\em The Finite Element Method for Elliptic Problems},
     North-Holland, Amsterdam, 1978.
 \bibitem{easymesh}
      {\sc B. Niceno},
      {\em EasyMesh  Version 1.4: A Two-Dimensional Quality Mesh Generator},
      \url{http://www-dinma.univ.trieste.it/nirftc/research/easymesh}.

    \bibitem{gan}
      {\sc X. Gan and J.E. Akin},
      {\em  Superconvergent second order derivative recovery technique and
      its application in a nonlocal damage mechanics model},
      Finite Elements in Analysis and Design,  35(2014), 118--127.
 \bibitem{huang2013}
     {\sc C. Huang and Z. Zhang},
     {\em Polynomial preserving recovery for quadratic elements on
	 anisotropical meshes},
	 Numer. Methods Partial Differential Equations, 28(2012), 966-983.
 \bibitem{huang2008}
     {\sc Y. Huang and J. Xu},
     {\em Superconvergence of quadratic finite elements on mildly structured
     grids.}, Math. Comp., 77(2008), 1253--1268.
 \bibitem{Kamenski2012}
     {\sc L. Kamenski and W. Huang},
     {\em How a nonconvergent recovered Hessian works in mesh adaptation},
     arXiv:1211.2877v2[math.NA].
 \bibitem{lakhany1998}
     {\sc A. M. Lakhany and J. R. Whiteman}
     {\em Superconvergent Recovery Operators: Derivative Recovery Techniques},
     Finite Element Methods: Superconvergence, Post-processing, and a posteriori
     estimates, M. Krizek, P. Neittaanmaki, and R. Stenberg, Marcel Dekker INC,
     Newyork, 1998, pp.~195--216.
 \bibitem{lakkis2011}
     {\sc O. Lakkis and T. Payer},
     {\em A finite element Method for second order nonvariational elliptic
     problems}, SIAM J. Sci. Comput., 33(2011), 786--801.
   \bibitem{lakkis2013}
     {\sc O. Lakkis, Omar and  T. Pryer},
     {\em A finite element method for nonlinear elliptic problems},
     SIAM J. Sci. Comput.,  35(2013),  2025--2045.
 \bibitem{naga2004}
     {\sc A. Naga and Z. Zhang},
     {\em A posteriori error estimates based on the polynomial preserving recovery},
     SIAM J. Numer. Anal., 42(2004), 1780--1800.
 \bibitem{naga2005}
     {\sc A. Naga  and  Z. Zhang}, {\em The polynomial-preserving recovery
     for higher order finite element methods in 2D and 3D},
      Discrete Contin. Dyn. Syst. Ser. B  5(2005), 769--798.
  \bibitem{neilan}
    {\sc M. Neilan}, {\em Finite element methods for fully nonlinear
      second order PDEs based on a discrete Hessian with applications
      to the Monge-Amp$\grave{\text{e}}$re equation},
      J. Comput. Appl. Math.,  263(2014), 351--369.
  \bibitem{nitsche}
    {\sc J. A. Nitsche and A. H.  Schatz},
    {\em Interior estimates for Ritz-Galerkin methods},
    Math. Comp.,  28(1974), 937--958.
 \bibitem{ovall2008}
     {\sc J.  Ovall},
     {\em Function, gradient, and Hessian recovery using quadratic edge-bump
     functions},
     SIAM J. Numer. Anal. 45(2007), 1064--1080.
 \bibitem{pouliot2013}
     {\sc B. Pouliot, M. Fortin, A. Fortin and E. Chamberland},
     {\em On a new edge-base gradient recovery technique},
     Int. J. Numer. Meth. Engng., 93(2013), 52--65.
 \bibitem{picasso2011}
     {\sc M. Picasso, F. Alauzet, H. Borouchaki, and P. George},
     {\em A numerical study of some Hessian recovery techniques on isotropic
     and anisotropic meshes},
     SIAM J. Sci. Computl., 33(2011), 1058--1076.
 \bibitem{vallet2007}
     {\sc M.-G. Vallet, C.-M. Manole, J. Dompierre, S. Dufour, and F. Guibault},
     {\em Numerical comparsion of some Hessian techniques},
     Internat. J. Numer. Methods Engrg., 72(2007), 987--1007.
 \bibitem{wahlbin1995}
     {\sc L.B. Wahlbin},
     {\em Superconvergence in Galerkin finite element methods},
     Lecture Notes in Mathematics, 1605, Springer-Verlag, Berkin, 1995.
 \bibitem{wu2007}
     {\sc H. Wu and Z. Zhang},
     {\em Can we have superconvergent gradient
     recovery under adaptive meshes ?}, SIAM J. Numer. Anal.,
     45(2007), 1701--1722.
 \bibitem{xu2004}
     {\sc J. Xu and Z. Zhang},
     {\em Analysis of recovery type a posteriori error estimators for
     mildly structured grids}, Math. Comp., 73(2004), 1139--1152.
 \bibitem{zhang2005}
     {\sc Z. Zhang and  A. Naga},
     {\em A new finite element gradient recovery method: superconvergence property},
     SIAM J. Sci. Comput., 26(2005), 1192--1213.
 \bibitem{zz1992}
     {\sc O.C. Zienkiewicz and J.Z. Zhu}, {\em The superconvergent patch
     recovery and a posteriori error estimates. I. The recovery technique},
     Internat. J. Numer. Methods Engrg., 33(1992), 1331--1364.
\end{thebibliography}
\end{document}